\documentclass[11pt,a4paper,twoside]{article} 

\usepackage{amsmath,amssymb,graphicx,xcolor}
\usepackage{mathtools,amsthm}
\usepackage[cal=pxtx]{mathalpha} 
\usepackage[utf8]{inputenc}
\usepackage[T1]{fontenc}
\usepackage[english]{babel}
\usepackage{framed,xcolor}
\usepackage{tikz,pgfplots}\pgfplotsset{compat=1.9}
\usepackage[hidelinks]{hyperref}
\usepackage{enumitem}

\usepackage[only,varodot]{stmaryrd} 
\setlist[itemize]{label={\footnotesize\textcolor{gray}{\textbullet}},leftmargin=15pt}
\usepackage[margin=2.5cm]{geometry}

\usetikzlibrary{pgfplots.groupplots}
\usepackage[labelfont={bf},font={small},skip=0pt]{caption}
\usepackage[skip=-10pt]{subcaption}
\theoremstyle{definition}
\newtheorem{definition}{Definition}[section]
\newtheorem{theorem}[definition]{Theorem}
\newtheorem{proposition}[definition]{Proposition}
\newtheorem{lemma}[definition]{Lemma}

\theoremstyle{definition}
\newtheorem{remark}[definition]{Remark}

\newtheorem{example}[definition]{Example}

\numberwithin{equation}{section}

\long\def\DETAILED#1\endDETAILED{\begin{shaded} #1\end{shaded}} 
\long\def\DETAILED#1\endDETAILED{} 
\definecolor{shadecolor}{gray}{0.95}
\definecolor{bleu}{rgb}{0,0.45,0.60}  
\definecolor{ForestGreen}{RGB}{18,90,18} 
\definecolor{amethyst}{rgb}{0.6, 0.4, 0.8} 

\def \disp {\displaystyle}
\newcommand*{\e}{\mathrm{e}}
\def \d {{\rm{d}}}
\def \X {{\mathbb X}} 
\def \N {{\mathbb N}} 
\def \R {{\mathbb R}}
\def \C {\mathcal{C}}
\def \M {\mathcal{M}}
\def \eps {{\varepsilon}}

\def \D {\mathcal{D}}

\usepackage{dsfont}
\newcommand{\un}{\mathds{1}}

\def \bv {{\bf v}}

\def \phi {{\varphi}}
\newcommand{\cm}{{\mathfrak{m}}} 
\newcommand \nrg {{\mathcal{E}}}

\newcommand{\vol}[1]{{\mathnormal{V}}\hspace{-1mm}{\mathnormal{ol}}\left(#1\right)}

\newcommand*{\grande}{{\mathring{x}}}   \newcommand{\piccola}{{\dot{x}}}
   
\newcommand{\bx}{{\bf x}} 
\newcommand{\by}{{\bf y}} 
\newcommand*{\bgrande}{{\mathring{\bx}}}   \newcommand{\bpiccola}{{\dot{\bx}}} 
\newcommand{\bpiccolaR}{{\dot{\bx}_{\vert_R}}}
\newcommand{\bgrandey}{{\mathring{\by}}}   \newcommand{\bpiccolay}{{\dot{\by}}}
\newcommand{\bgrandev}{{\mathring{\bv}}}   \newcommand{\bpiccolav}{{\dot{\bv}}}
\newcommand{\grandev}{{\mathring{v}}}  \newcommand{\piccolav}{{\dot{v}}}

\newcommand*{\XGrande}{{\mathring{\X}}}
\newcommand*{\XPiccola}{{\dot{\X}}}
\newcommand*{\MGrande}{{\mathring{\M}}}
\newcommand*{\MPiccola}{{\dot{\M}}}
\def \rgrande {\mathring{r}}
\newcommand{\mesGrande}{{\lambda}} 
\newcommand{\BGrande}{{\mathbb{B}}} 
\def \rpiccola {\dot{r}} 
\def \rdep {{\overset{\varodot}{r}}}
\newcommand{\Rgrande}{\mathring{R}}
\newcommand{\Rpiccola}{\dot{R}}
\newcommand*{\Grande}{\mathring{X}}
\newcommand*{\Piccola}{\dot{X}}

\newcommand*{\sigmaG}{
}
\newcommand*{\sigmaP}{\dot{\sigma}}

\newcommand*{\zpiccola}{{\dot{z}}}
\newcommand*{\phidep}{{\phi}}
\newcommand*{\psiG}{{\mathring{\psi}}}
\newcommand*{\psiPR}{{\dot{\psi}^R}}

\newcommand*{\WGrande}{{\mathring{W}}}
\newcommand*{\WGrandei}{{\mathring{W_i}}}
\newcommand*{\WGrandeone}{{\mathring{W_1}}}
\newcommand*{\WPiccola}{{\dot{W}}}
\newcommand*{\WPiccolak}{{\dot{W_k}}}

\newcommand*{\mugrande}{{\mathring{\mu}}}
\newcommand*{\numR}{{\bar{\nu}^{m,R}}}

\newcommand{\BadPath}{{\mathcal{B}}}
\newcommand{\NicePath}{{\mathcal{N}}}
\newcommand{\ZRmixing}{{\mathcal{Z}_R}}

\newcommand*{\Vovlap}{\mathcal{V}_{\rm{ovlap}}}
\newcommand*{\rspace}{\,\hspace{-.14em}}

\newcommand{\defeq}{\vcentcolon=}
\newcommand{\eqdef}{=\vcentcolon}

\author{Myriam Fradon\thanks{Univ. Lille, CNRS UMR 8524, Laboratoire P. Painlevé, {\tt myriam.fradon@univ-lille.fr}}\ \quad Julian Kern\thanks{Weierstrass Institute for Applied Analysis and Stochastics, Berlin, {\tt kern@wias-berlin.de}}\ \quad Sylvie R\oe lly\thanks{University of Potsdam, {\tt roelly@math.uni-potsdam.de}}\ \quad Alexander Zass\thanks{Weierstrass Institute for Applied Analysis and Stochastics, Berlin, {\tt zass@wias-berlin.de}}}
 \date{\small \emph{Francis, un bon, un vrai copain des temps anciens, lorsque nous défrichions ensemble, avec Claude, Sylvie M., Christian et Patrick  le domaine des} systèmes de particules \emph{nouvellement né}.\\ 
 \emph{Travail de fourmi, sans faire de mousse. \\
 Francis était toujours présent, constant, confiant, clairvoyant et gentil.\\
 Copain idéal pour ne pas perdre le Nord. Copain qui reste patient quand on ne l'est plus. \\
 Toutes les qualités d'un} honnête homme \emph{et d'un scientifique de haut vol.\\
 Merci à toi Francis, fidèle éclaireur. Tu es parmi nous, tout simplement. \emph{(Sylvie)}
 } }

\title{Diffusion dynamics for an infinite system of two-type spheres and the associated depletion effect}

\usepackage[hang,flushmargin]{footmisc}

\usepackage[switch,columnwise]{lineno}
\usepackage{lipsum} 

\makeatletter
\let\@fnsymbol\@arabic
\makeatother

\mathtoolsset{showonlyrefs}

\begin{document}
\maketitle

\begin{abstract}
We consider a random diffusion dynamics for an infinite system of hard spheres of two different sizes evolving in $\R^d$, its reversible probability measure, and its projection on the subset of the large spheres. The main feature is the occurrence of an attractive short-range dynamical interaction -- known in the physics literature as a depletion interaction -- between the large spheres, which is induced by the hidden presence of the small ones. By considering the asymptotic limit for such a system when the density of the particles is high, we also obtain a constructive dynamical approach to the famous discrete geometry problem of maximising the contact number of $n$ identical spheres in $\R^d$. As support material, we propose numerical simulations in the form of movies.

\medbreak
\end{abstract}

\section{Introduction: the model and its configuration space}\label{Introduction}
Consider hard spheres randomly oscillating in a bath of very small particles (see Figure~\ref{fig:twosizemodel})  which are themselves independently randomly vibrating. As soon as two hard spheres are very close, one can  observe the appearance of a strong mutual attraction which forces them to stay close together for a certain random amount of time. As no external force is acting on the system, this is a quite surprising phenomenon. What is going on?

To provide an answer, we propose a mathematical formulation of this apparent paradox: in the introductory section, we first present the classical Asakura--Oosawa model from Chemical Physics, and then describe the mathematical setting of this work. 
\subsection{The origin of the model: a short heuristic}\label{sect_model}
The model is the following: 
spheres of equal radius $\rgrande$ evolve in a bath of much smaller ones with radius $\rpiccola \ll \rgrande$. The radius $\rpiccola$ is called {\em depletion radius} for a reason which will become clear later.
The larger spheres are \emph{hard} in the sense that they cannot overlap: their interiors must always stay disjoint.
The smaller spheres -- called {\em particles} for clarity -- which compose the (random) medium  are also not allowed to overlap the large ones. This leads to the presence of a virtual spherical shell around each large sphere, corresponding to the zone in which the centres of the particles are not allowed, see Figure~\ref{fig:twosizemodel} (left). This zone, called {\em depletion shell}, will play a fundamental role in what follows.
Finally, because of their small radii, the particles can be viewed as an \emph{ideal gas}, i.e., they can overlap each other.
In Figure \ref{fig:twosizemodel}, a realisation of this  model; the region coloured in orange is the union of the depletion shells.

\begin{figure}
\begin{center}
\includegraphics[height=5.5cm]{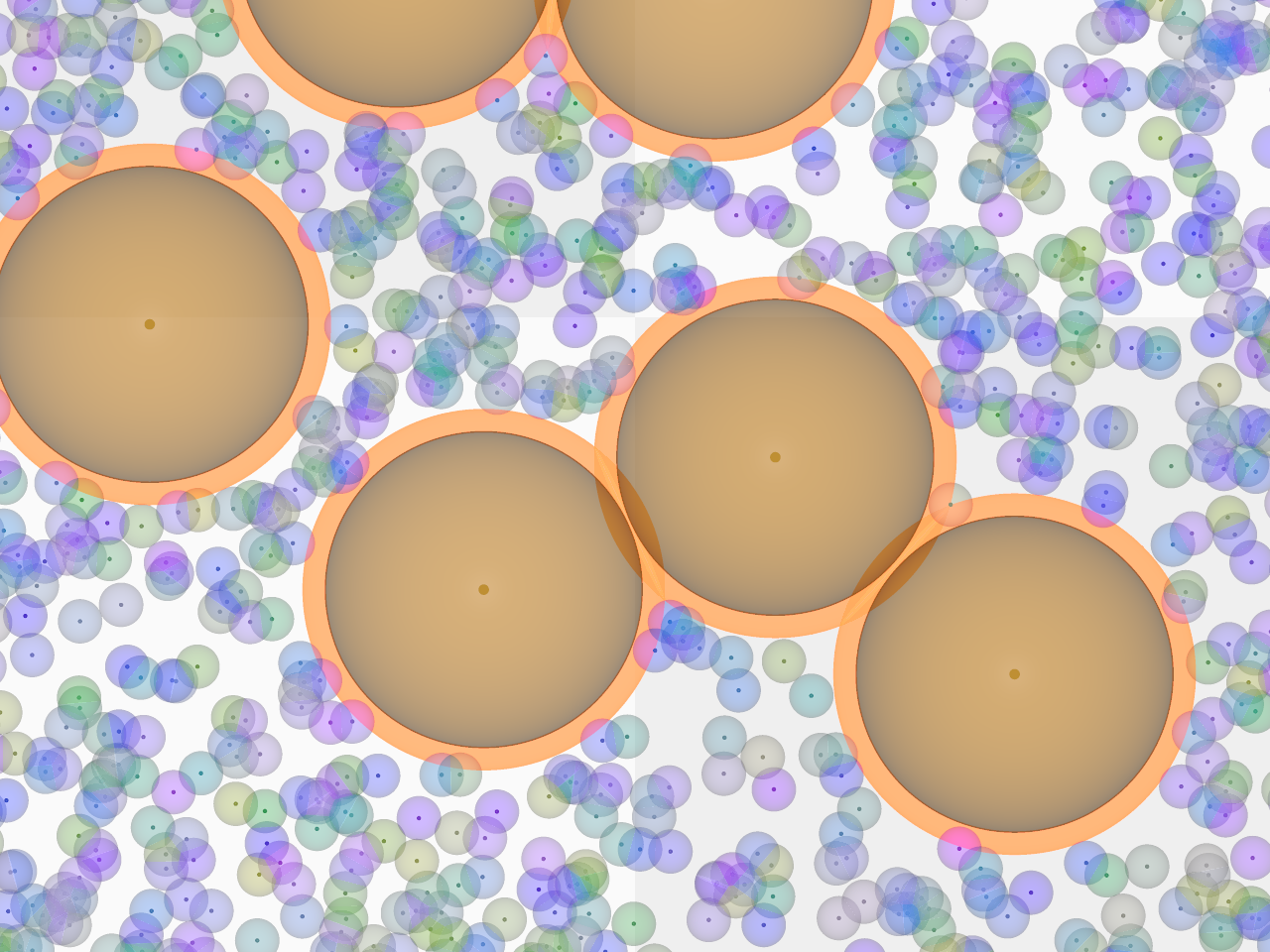}
\hspace{1em}
	\includegraphics[height=5.5cm]{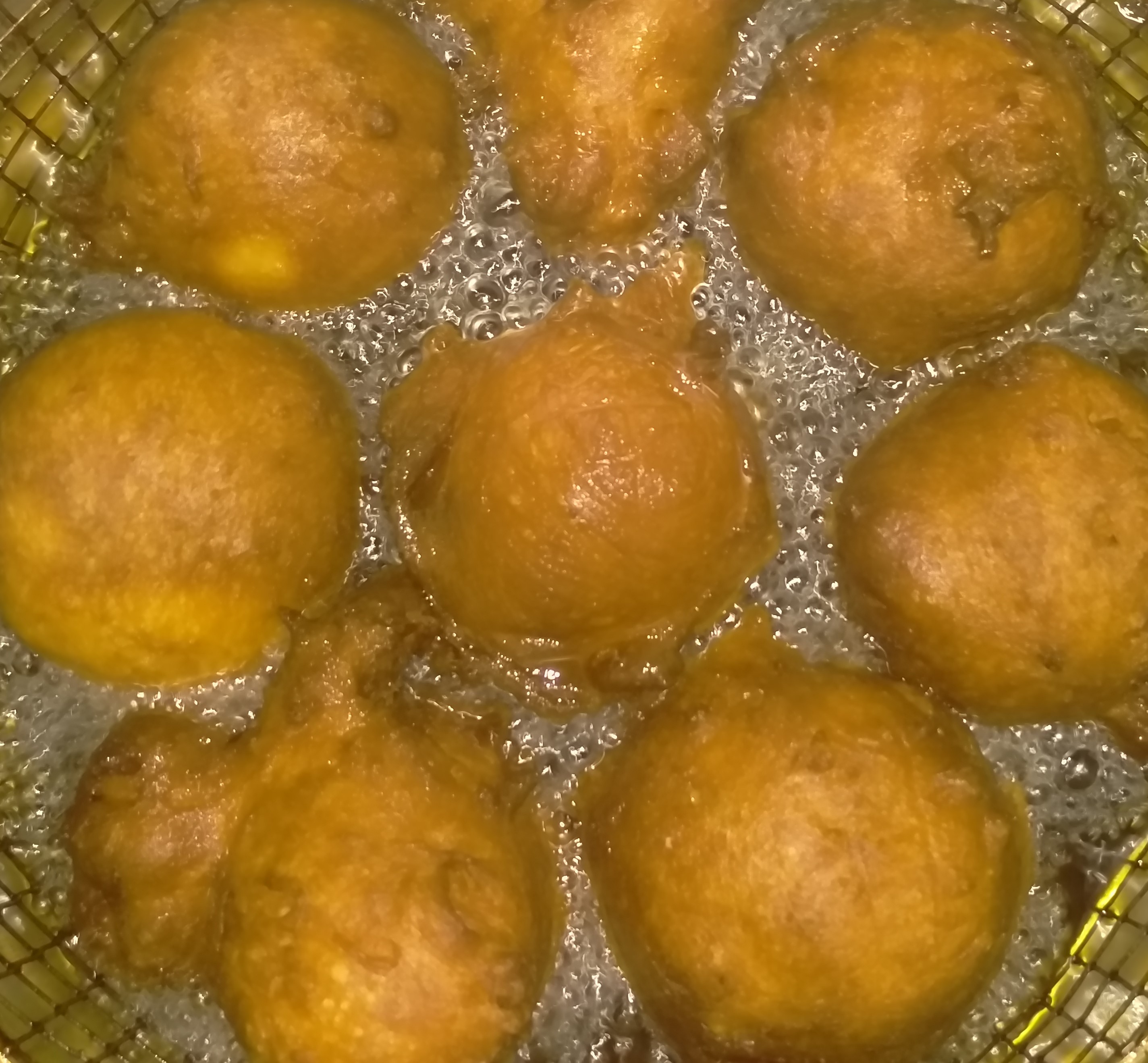}
\end{center}
\caption{{\small Identical hard spheres in a bath of 
identical small particles. 
An ideal mathematical representation (left) and a culinary realisation (right, jelly doughnuts in  particles of frying oil).\\
Left: the orange depletion shells around the brown spheres overlap.}}\label{fig:twosizemodel}
\end{figure}

This model is known in the physics literature as AO-model in reference to the seminal work of S.~Asakura and F.~Oosawa, who introduced in~\cite{AO54,AO58} a size-asymmetric binary mixture in the Euclidean space $\R^3$ to describe colloids (large spheres) in a bath (or emulsion) of ideal polymers (the small particles) in the context of Chemical Physics, see also~\cite{Vr76}. The reader may refer to~\cite{Ma18} for a clear overview of the physical phenomenon and its modelling. Its importance is underlined by Binder, Virnau and Statt in~\cite{BVS14}: ``Since 60 years the Asakura--Oosawa model, which simply describes the polymers as ideal soft spheres, is an archetypal description for the statistical thermodynamics of such systems, accounting for many features of real colloid-polymer mixtures very well.'' Indeed the bath of polymers induce a new attractive interaction between the colloids, called effective or {\em depletion interaction}. For a physical theoretical analysis of this general phenomenon see, e.g.,~\cite{RED00} and the valuable monograph~\cite{LT11}.
As an illustration of the effect of the depletion force one can cite the ordered, helical conformation of long molecular chains like DNA in the Euclidean space $\R^3$ by interpreting this geometric structure as being thermodynamically induced by the entropy minimisation of depleting spheres, see~\cite{SK05}.

A first rigorous mathematical treatment of the AO-model for an infinite number of both sphere types with methods of Statistical Mechanics appeared only recently in a series of papers by S.~Jansen and coauthors, see~\cite{JT19, JT20, WJL22}.
\subsection{The mathematical model}\label{sec:mathmodel}

The geometric objects we deal with in this paper are spheres of two different types: the \emph{hard spheres} with fixed radius $\rgrande$ and the \emph{particles} with radius $\rpiccola<\rgrande$. They are identified by the position of their centres in $\R^d$; if a point $x \in \R^d$ is  the centre of a hard sphere we denote it by $\grande$, if it is the centre of a particle we denote it by $\piccola$. 
In this way we can consider the set $\X = \XGrande \bigsqcup \XPiccola \simeq \R^d \times\{ \circ,\cdot{~} \}$ as duplication of $\R^d$, to distinguish between the two types of spheres.
\smallskip

Throughout the paper, the number of hard spheres will be finite and fixed, equal to $n\geq 1$.
\smallskip

The \emph{configuration space} of the system is the set $\M$ of $\sigma$-finite Radon point measures on $\X$, i.e., those of the form
 \begin{equation*}
    \bx = \bgrande + \bpiccola = \sum_{i=1}^n \delta_{\grande_i} + \sum_{k \in K}\delta_{\piccola_k},\quad \grande_i \in \XGrande,\, \piccola_k \in \XPiccola,\quad K\subset \N^*,
 \end{equation*}
such that for any compact $\Lambda\subset\X$, $\bx(\Lambda) < + \infty$. 
\smallbreak

With this formalism, $\bgrande$ denotes the point measure of centres of hard spheres belonging to the configuration $\bx$ and $\bpiccola$ denotes the point measure of centres of the particles in $\bx$. 
As the point measures we consider are a.s. simple, we can use the notation $\bx$ interchangeably for the point measure or for its support $\{\grande_i,\piccola_k, 1\le i \le n, k \in K\}\subset \R^d$. We write the sum of two point measures as the juxtaposition $\bgrande \bpiccola\defeq \bgrande + \bpiccola$.
For $m \in \N$, let  $\M_{m} \subset \M$ be the set of finite configurations with exactly $m$ particles, that is
\begin{equation*}
    \M_{m}\defeq \{\bx=\bgrande\bpiccola \in \M,  \ \bpiccola (\XPiccola) =m\}.
\end{equation*}
It is used in the first step of the proof of Theorem \ref{th:existenceSninfty} to approximate the infinite configurations. 
We write $\MGrande$ (resp.~$\MPiccola$) for the point measures supported only by hard spheres (resp.~particles). 

In the following, $B(y,r)$ denotes the closed ball in $\R^d$ centred in $y\in\R^d$ with radius $r\in\R_+$.

\bigskip

The so-called \emph{admissible} configurations, i.e., those respecting the non-overlap constraint, make up the following subset $\D\subset \M$:
\begin{equation} \label{eq:D}
\D = \Bigg\{\bx=\bgrande\bpiccola \in \M:
\begin{array}{rl}
    \forall i \neq j, \ 
	 &|\grande_i - \grande_j| \ge 2\rgrande, \\
    \forall i , k, \ 
	 &|\grande_i - \piccola_k| \ge \rgrande+\rpiccola
\end{array}
\Bigg\}.
\end{equation}
The second type of constraints in \eqref{eq:D} can be interpreted as follows: around each hard sphere $B(\grande_i,\rgrande)$ there is a shell of thickness $\rpiccola$, called \emph{depletion shell}, that is forbidden for the centres of the particles $(\piccola_k)_k$, see Figure~\ref{fig:twosizemodel}. We therefore introduce the radius $\rdep$, seen as an enlargement of the original hard-sphere radius $\rgrande$: 
\begin{equation} \label{eq:defrho}
\rdep\defeq \rgrande+\rpiccola =
 \rgrande \, (1 + \rho), \textrm{ where } \rho\defeq \frac{\rpiccola}{\rgrande}\in [0,1[\, ,
\end{equation}
and identify the forbidden area for particles around the admissible configuration of hard spheres $\bgrande$ as the interior of
\begin{equation}\label{UnionBouees}
  \BGrande(\bgrande)\defeq  \bigcup_{i=1}^{n}B(\grande_i,\rdep) \subset \R^d.
\end{equation}
As we will see in the paper, the relative size $\rho$ between particle and hard sphere radii, defined in \eqref{eq:defrho}, plays an important role in the study of this two-size model.

\bigbreak
Our aim here is to present and study a dynamical version of the AO-model and its depletion feature. We first construct, in Section \ref{sect_dyninfinie}, infinite-dimensional random diffusion dynamics whose reversible (i.e., equilibrium) measure is the AO-Model for $n$ hard spheres in a bath of infinitely many particles. Section \ref{sec:depletion} is devoted to the study of the projection of this two-type reversible  measure onto the subsystem of hard spheres.
We first notice how it induces a new attractive interaction (in the sense of Statistical Mechanics) between the hard spheres, called depletion interaction. This new term is induced by the hidden presence of the particle bath and is proportional to the volume of the depletion shells around the hard spheres, see Proposition \ref{prop:tracemesureequ}. Detailed computations and geometric comments in particular cases are then presented. Moreover, a gradient random dynamics associated to this measure is proposed in Section \ref{sec:dep_dynamics}. In Section \ref{sec:packing}, we consider the asymptotic regime corresponding to the system of $n$ spheres in a bath with a very high density of particles. The depletion interaction thus dominates the system to the extent that the reversible measure concentrates on $n$ hard sphere configurations in $\R^d$ which maximise their contact number. In this way, we obtain a constructive random dynamical approach via gradient diffusions to the difficult problem of optimal sphere packing for any number $n$ of spheres and in any dimension $d$. 

In order to get an understanding for the behaviour of the two-type dynamics of Section \ref{sect_dyninfinie}, as well as of the gradient random dynamics with depletion studied in Section \ref{sec:dep_dynamics}, we decided to write code for simulations. The link to the GitLab page is provided in Section \ref{AppendixSimulations}, along with a short presentation of the animations one can find there.

\section{Diffusion of hard spheres in an infinite bath of Brownian particles}\label{sect_dyninfinie}
Finding appropriate random dynamics that describe the time evolution of various two-type physical systems is an old challenge. See, e.g., 
the mechanical model of Brownian motion proposed in \cite{DGL} for the motion of a large component whose velocity follows an Ornstein-Uhlenbeck diffusion in an infinite bath of small particles; 
\cite{ST93}, in which a Brownian sphere interacts with infinitely-many particles of vanishing radius; 
\cite{BurdzyCP}, in which the authors exhibit a kind of Archimedes’ principle for a large disc evolving as a Brownian motion with drift (due to the force of gravity) in a one-sided open cylinder of $\R^2$, submerged in a large number of much smaller discs. 

Despite the vast literature, to the best of our knowledge, there is no study that takes into account a two-type hard-core interaction. The specificity of our approach lies therefore in the construction of a strong solution to an infinite-dimensional stochastic differential system for a two-size model of large hard spheres and small particles that are diffusing under the infinitely many non-overlap constraints \eqref{eq:D}. The main technical difficulty consists in controlling the reflection at the boundary of the set of admissible configurations, expressed mathematically as infinitely many local-time terms appearing in the stochastic differential equation (SDE) that describes the time evolution of each sphere.

\subsection{Existence and uniqueness result for an infinite-dimensional 
random dynamics with reflection}

We now introduce and study the random evolution in $\R^d$ of our two-type system. To simplify, we restrict the time evolution to the time interval $[0,1]$, noting that it can be extended by Markovianity to any time interval.

The system is described as follows:
\begin{itemize}
    \item $n$ {\em hard spheres} with radius $\rgrande$, whose centres at time $t$ are denoted by $\{\Grande_1(t), \dots, \Grande_n(t)\}$, move according to $n$ independent Brownian motions. 
    
In order to avoid their dispersion at infinity, they are smoothly confined around the origin by a self-potential $\psiG: \R^d \to \R $ of class $\C^2$ with bounded derivatives and satisfying
\begin{equation} \label{eq:psiG}
 \int_{\R^d} \exp \big(-\psiG (x)\big) \, dx = 1 \quad \text{ and } \quad \int_{\R^d} |x|^2 ~ \e^{-\psiG(x)} \, d x <+\infty .
 \end{equation}
It is simple to show that such a function (having linear growth with respect to the Euclidean norm at infinity) exists.
Moreover, the measure 
\begin{equation*}
    \lambda(dx) \defeq \exp \big(-\psiG (x)\big) \, dx
\end{equation*}
is a probability measure with second moment, and plays a reference role in what follows.
    \item The hard spheres evolve in a time-inhomogeneous {\em random medium} consisting of a field $\sum_k \delta_{\Piccola_k(\cdot)}$ of intensity $\zpiccola$ of infinitely many small particles, themselves moving according to $\sigmaP$-scaled independent Brownian motions.
    \item The only interactions between the hard spheres and the small particles are due to the {\em non-overlap constraints} \eqref{eq:D}, in the sense that, at each time, the two-type configuration should be admissible.
\end{itemize}
Fix a probability space $(\Omega,\mathcal{F},P)$. We can then describe this two-type dynamics with the following infinite-dimensional SDE with reflection:
\begin{equation}\label{eq:Sninfty} \tag{${\mathfrak S}$}
\begin{cases}
\begin{array}{l}
 \textrm{for } i,j \in \{1,\dots,n\}, \,   k \in \N^*, \,  t \in [0,1] ,                             \\ \disp  
 \Grande_i(t) = \Grande_i(0) +\sigmaG\, \WGrande_i(t) - \frac{1}{2} \int_0^t \sigmaG \ \nabla\psiG \big(\Grande_i (s)\big)\, ds \\  \disp
 \phantom{d \Grande_i (t) = }
                          + \sum_{j=1}^n \int_0^t \big(\Grande_i(s)-\Grande_{j}(s)\big)  d L_{ij}(s) + \sum_{k\geq 1} \int_0^t \big(\Grande_i(s)-\Piccola_{k}(s)\big)  d\ell_{ik}(s) , \\ \disp 
 \Piccola_k(t) = \Piccola_k(0) +\sigmaP\, \WPiccola_k(t) + \sigmaP^2 \sum_{i=1}^n \int_0^t \big(\Piccola_k(s)-\Grande_{i}(s)\big) \, d\ell_{ki}(s) , \\ \disp
  L_{ij}(0) = 0 , \quad L_{ij} \equiv L_{ji}, \quad
  L_{ij}(t) = \int_0^t \un_{|\Grande_i (s)-\Grande_j(s)|=2 \,\rgrande} \, d L_{ij}(s), \quad L_{ii} \equiv 0 , \\ \disp
   \ell_{ik}(0) = 0 , \quad \ell_{ik} \equiv \ell_{ki}, \quad
  \ell_{ik}(t) = \int_0^t \un_{|\Grande_i (s)-\Piccola_k(s)|=\rgrande + \rpiccola} \, d \ell_{ik}(s), \quad \ell_{ii} \equiv 0,
  \end{array}
\end{cases}
\end{equation}
 where the i.i.d.~sequences of $\R^d$-valued Brownian motions $(\WGrande_i)_{i=1,\dots,n}$ and $ (\WPiccola_k)_{k\in \N^*}$ are independent.
 
The local times $(L_{ij})_{i,j \in \{1,\dots,n\}}$ ensure that the hard spheres do not overlap pairwise. 
In case of a collision, they are submitted to an instantaneous repulsion corresponding to a normal reflection at the boundary of the set of admissible configurations.
Analogously, the local times $(\ell_{ik})_{i\in \{1,\dots,n\},\, k \in \N^*}$ ensure that the small particles do not overlap with the hard spheres.

The gradient term $\nabla\psiG $ guarantees that, in the absence of small particles, the large spheres undergo a recurrent diffusive motion whose unique reversible probability measure is known. 
The diffusion coefficient $\sigmaP$ parametrises the mobility of the small particles.
\smallbreak

We can now state the main result of this section. 
\begin{theorem} \label{th:existenceSninfty}
     The infinite-dimensional SDE with reflection \eqref{eq:Sninfty} admits for $\mu$-almost every deterministic initial condition a unique $\D$-valued strong solution, where the probability measure $\mu$, concentrated on $\D$, is given by \eqref{eq:mu}.
\end{theorem}
   
The rest of this section is devoted to the proof of the above theorem. We split it into four steps, taking inspiration from, and generalising, the existence theorem obtained in \cite{FR2} for an infinite-dimensional diffusion with reflection of equal (one-size) spheres:
\medbreak

\noindent 
{\bf Step 1:} \emph{Dynamics for $n$ hard spheres and $m$ confined particles.}

We first approximate the above infinite-dimensional dynamics by a two-type dynamics concerning only a finite number $m\geq 1$ of particles. 
Moreover, we confine them by adding to their dynamics a restoring gradient drift that prevents their dispersion. The resulting dynamics is then described by the following finite-dimensional SDE:
\begin{equation}\tag{${\mathfrak S}_{m,R}$}\label{eq:SnmR}
\begin{cases}
\begin{array}{l}
  \textrm{for } i, j \in \{1,\dots,n\}, \,  k \in \{1,\dots,m\}, \,  t \in [0,1] ,                             \\ \disp
  \Grande_i (t) = \Grande_i(0) +\sigmaG\, \WGrande_i(t) -\frac{1}{2} \int_0^t \sigmaG \ \nabla\psiG \big(\Grande_i (s)\big)\, d s \\  
  \disp \phantom{d\Grande_i (t) = \sigmaG\,}
                   + \sum_{j=1}^n \int_0^t (\Grande_i-\Grande_{j})(s) dL_{ij}(s) + \sum_{k=1}^{m} \int_0^t (\Grande_i-\Piccola_{k})(s) d\ell_{ik}(s) , \\ \disp 
  \Piccola_k(t) = \Piccola_k(0) +\sigmaP\, \WPiccola_k(t) - \frac{\sigmaP^2}{2} \int_0^t \sigmaG \ \nabla\psiPR\big(\Piccola_i (s)\big)\, d s 
  \\  
  \disp \phantom{\Piccola_k(t) = \Piccola_k(0) +\sigmaP\, \WPiccola_k(t) - \frac{\sigmaP^2}{2}}
                  + \sigmaP^2 \sum_{i=1}^n \int_0^t (\Piccola_k-\Grande_{i})(s) d\ell_{ki}(s) ,\\ \disp
  L_{ij}(0) = 0 , \quad L_{ij} \equiv L_{ji}, \quad
  L_{ij}(t) = \int_0^t \un_{|\Grande_i (s)-\Grande_j(s)|=2\rgrande} \, d L_{ij}(s), \quad L_{ii} \equiv 0 , \\ \disp
   \ell_{ik}(0) = 0 , \quad \ell_{ik} \equiv \ell_{ki}, \quad \ell_{ik}(t) = \int_0^t \un_{|\Grande_i (s)-\Piccola_k(s)|=\rgrande + \rpiccola} \, d \ell_{ik}(s), \quad \ell_{ii} \equiv 0 .
  \end{array}
\end{cases}
\end{equation}
where $\WGrande_i,1 \leq i \leq n, \WPiccola_{k},1 \leq k \leq m,$ are independent $\R^d$-valued Brownian motions.\\[2mm]

The function $\psiPR: \R^d \rightarrow \R_+$ confining the particles is of class $\C^2$ with bounded derivatives. Moreover, it depends on the parameter $R\in \N$ in the following way:
\begin{equation} \label{eq:psiPR}
    \psiPR(x)=0 \textrm{ if } x \in B(0,R) \quad \textrm{ and } \quad \sum_{R=1}^\infty \int_{B(0,R)^c} \e^{-\psiPR(x)} \, d x <+\infty.
\end{equation}
Since it vanishes in the ball $B(0,R)$, its confining effect decreases and eventually disappears as $R$ tends to infinity.
Such a function can be constructed, e.g., by defining it proportional to $|x|-R $ for $x$ far from the origin.

We define the set of \emph{finite} admissible configurations with $m$ particles as \mbox{$\D_m \defeq \D \cap \M_{m}$}.

\begin{proposition} \label{prop:existencesolrevSnmR}
The SDE with reflection \eqref{eq:SnmR} admits, for any deterministic initial condition in the interior of the domain $\D_m$, a unique strong solution 
\begin{equation*}
   \big( \Grande_i^{m,R}(t), \Piccola_k^{m,R}(t),  L_{ij}^{m,R}(t), \ell_{ik}^{m,R}(t) \big)_{t\in [0,1],\, 1\le i,j \le n,\, 1\le k\le m} \,.
\end{equation*}
The finite measure $\nu^{m,R}$, concentrated on the admissible configurations with $m$ particles, and given by
\begin{equation}\label{eq:numR}
    \nu^{m,R} (d\bx) 
 \defeq \un_{\D_m}(\bx) ~   \e^{- \sum_{k=1}^m \psiPR(\piccola_k)} ~ \otimes_{i=1}^n \lambda (d\grande_i) \, \otimes_{k=1}^m d\piccola_k ,
\end{equation}
where $\bx=\{\grande_1,\dots,\grande_n,\piccola_1,\dots,\piccola_m\}$,
is reversible for the dynamics \eqref{eq:SnmR}.
\end{proposition}
\begin{remark} \label{rem:psiPR}
Due to the assumptions \eqref{eq:psiPR} satisfied by the confining potential $\psiPR$, the measure $\nu^{m,R}$ is mainly supported on configurations whose $m$ particles are in or close to the ball $B(0,R)$. Moreover, the integral $\int_{\R^d} \e^{-\psiPR(x)} dx $  is finite and increases at most polynomially in $R$ when $R$ tends to infinity.  
\end{remark}

\begin{proof}[Proof of Proposition \ref{prop:existencesolrevSnmR}]
The system \eqref{eq:SnmR} describes the dynamics of an $(n+m)d$-di\-men\-sio\-nal gradient diffusion with reflection at the boundary of the domain $\D_m$. The different size of the hard spheres and the particles induces a new geometric complexity which did not exist in the case of identical spheres studied in~\cite{FR2}.
\smallbreak
In \cite{MultipleConstraint}, the first author solved the question of existence and uniqueness of a reflected diffusion in a geometric domain whose boundary is induced by several constraints. There, the required assumptions are (i) a regularity condition on each constraint; and (ii) a so-called {\em compatibility} condition between the constraints. We check here that the domain $\D_m \subset \M_{m} $ satisfies such properties, as stated in \cite[Definition 2.1]{MultipleConstraint}.
\smallbreak

The interior of the domain $\D_m$ can be described as the following intersection of sets:
\begin{equation*}
    \textrm{int}(\D_m) = 
    \bigcap_{\substack{1\le i,j \le n\\
    j\not=i}}
    \Big\{\bx\in \M_m : \, \Gamma_{ij}(\bx) >0 \Big\} 
    \cap
    \bigcap_{\substack{1\le i \le n\\
     1 \leq k\leq m}}
     \Big\{ \bx\in \M_m: \, \gamma_{ik}(\bx) >0 \, \Big\} . 
\end{equation*}
The ${\mathcal C}^2$ constraint function 
\begin{equation}\label{def:Gammaij}
\disp \Gamma_{ij}(\bx)= \Gamma_{ij}(\bgrande)\defeq \frac{|\grande_i - \grande_j|^2}{ 4 \rgrande\rspace ^2} - 1
\end{equation}
controls the distance between the hard spheres $i$ and $j$, 
whereas the ${\mathcal C}^2$ constraint function 
\begin{equation*}
    \gamma_{ik}(\bx) \defeq \frac{|\grande_i - \piccola_k|^2 }{\rdep\rspace^2}-1
\end{equation*}
controls the distance between the hard sphere $i$ and the particle $k$. The boundary of $\D_m$, then, is a union of smooth boundaries, each one being the set of zeros of one constraint.\\[2mm]
\noindent
(i) \emph{Boundedness of first and second derivatives of the constraint functions.}

We first prove that the norm of the gradient of each constraint function is uniformly bounded from below on its induced boundary. 

It is straightforward to check that, for $\bx=\{\grande_1,\cdots,\grande_n,\piccola_1,\cdots,\piccola_m\}$, 
\begin{itemize}
    \item if $\Gamma_{ij}(\bx)=0$, i.e., $|\grande_i - \grande_j|^2 = 4\,\rgrande\rspace ^2$, then   
$\disp |\nabla \Gamma_{ij}(\bx)|^2 = 8\,|\grande_i - \grande_j|^2 /16 \rgrande\rspace ^4= 2/\rgrande\rspace ^2 >0$;
    \item if $\gamma_{ik}(\bx)=0$, i.e., $|\grande_i - \piccola_k|^2 = \rdep\rspace^2$ then  
$\disp |\nabla \gamma_{ik}(\bx)|^2 = 8/\rdep\rspace ^2 >0.$
\end{itemize}
Second derivatives are uniformly bounded from above because they are constant. \\[2mm]
\noindent
(ii) \emph{Compatibility between the constraints.}

We are now looking for a positive constant $b_0$ such that, at any point $\bx$ of the $ij$-boundary (resp.~$ik$-boundary) of $\D_m$,  there exists a non-zero vector $\bv\in (\R^d)^{n+m}$, such that
\begin{equation*} 
   \bv\cdot\nabla \Gamma_{ij}(\bx) \ge b_0 \,|\bv|~|\nabla \Gamma_{ij}(\bx)| \quad 
   (\textrm{resp.~} \bv\cdot\nabla \gamma_{ik}(\bx) \ge b_0 \,|\bv|~|\nabla \gamma_{ik}(\bx)|) .
\end{equation*}
More precisely, if a configuration $\bx=\bgrande\bpiccola$ belongs, e.g., to the $ij$-boundary, the hard spheres $i$ and $j$ collide. Heuristically, the vector $\bv$ indicates the most effective impulse for the configuration to come back into the interior of the domain $\D_m$, i.e., for the colliding spheres to get away from each other as fast as possible. The compatibility condition requires that the maximum angle between all these impulses, which is equal to $2\arccos{b_0}$, remains bounded away from $\pi$.

Fix $\bx \in\partial\D_m$, and let $C(\grande_i)$ be the  cluster around the $i$-th sphere $\grande_i$ (resp.~$C(\piccola_k)$ the cluster around the $k$-th particle $\piccola_k$), that is, the set of all spheres or particles of $\bx$ either touching $\grande_i$ or belonging to a chain of spheres and/or particles in contact including $\grande_i$ (similarly for $\piccola_k$). We define the centre of mass of such clusters by 
\begin{equation*}
\cm (\grande_i) \defeq \frac{1}{\# C(\grande_i)} \sum_{x_j\in C(\grande_i)} x_j 
\quad \textrm{ and } \quad
\cm (\piccola_k) \defeq \frac{1}{\# C(\piccola_k)} \sum_{x_j\in C(\piccola_k)} x_j
\end{equation*}
and the vector $\bv=\bgrandev\bpiccolav \in \R^{dn} \times \R^{dm}$ by
\begin{equation*}
\grandev_i\defeq\grande_i- \cm (\grande_i) 
\quad \textrm{ and } \quad 
\piccolav_k \defeq \piccola_k- \cm (\piccola_k).
\end{equation*}
It is not difficult to show that 
\begin{equation*}
\Gamma_{ij}(\bx)=0 \ \Rightarrow \ \bv\cdot\frac{\nabla \Gamma_{ij}(\bx)}{|\nabla \Gamma_{ij}(\bx)|}= \sqrt{2}\rgrande \quad
\textrm{ and }
\quad \gamma_{ik}(\bx)=0 \ \Rightarrow \ \bv\cdot\frac{\nabla \gamma_{ik}(\bx)}{|\nabla \gamma_{ik}(\bx)|}=\frac{\rdep}{\sqrt{2}}.
\end{equation*}
Moreover, $|\bv|^2 \le 4 \rgrande\rspace ^2(n+m)^3$. Therefore, one can choose
$ b_0 \defeq \displaystyle\frac{\rdep}{2\sqrt{2}~ \rgrande (n+m)^{3/2}} >0.$

Note that $b_0$ vanishes as $m$ tends to infinity. 
Therefore, this method cannot be applied to prove existence in the case $m=+\infty$.

\smallbreak
Having proved that the constraints defining the domain $\D_m$ are compatible in the sense of \cite{MultipleConstraint}, we can now apply Theorem 2.2 therein, yielding the existence and uniqueness, for each initial admissible configuration, of a strong solution to the SDE with reflection \eqref{eq:SnmR}. 

Its diffusion matrix is the $\big(d(n+m) \times d(n+m)\big)$-block matrix of diagonal matrices 
$\begin{pmatrix} \sigmaG I_{nd} &   0   \\
                                    0    & \sigmaP I_{md} \end{pmatrix} $ 
and its drift is given by the gradient of the potential function \mbox{$ \Phi(\bx) 
       \defeq \sum_{i=1}^n \psiG(\grande_i) + \sum_{k=1}^m \psiPR(\piccola_k) $}.

Applying \cite[Theorem 2.5]{MultipleConstraint} (see also \cite{ST87}), we get that $\un_{\D_m}(\bx) \e^{-\Phi(\bx)} d\bx$ is a time-reversible measure for the dynamics \eqref{eq:SnmR}. This concludes the proof of the proposition.
\end{proof}
\medbreak

\noindent 
{\bf Step 2:} \emph{Localisation of the initial particles.}

We consider again the finite-dimensional dynamics \eqref{eq:SnmR}, and fit the number $m$ of particles to the confinement parameter $R$ in the following way.
Let $\bx = \bgrande\bpiccola \in \D$ be an admissible configuration; we define the finite-dimensional process  
\begin{equation} \label{eq:XxR}
X^{\bx,R} = \left( \Grande_i^{\bx,R}(t), \Piccola_k^{\bx,R}(t),  L_{ij}^{\bx,R}(t), \ell_{ik}^{\bx,R}(t) 
                   \right)_{t\in [0,1],\, 1\le i,j \le n,\, 1\le k\le m} 
\end{equation} 
as the solution of the SDE \eqref{eq:SnmR} with initial condition $\bgrande$ for the $n$ hard spheres and $\bpiccolaR$ for the particles, where $\bpiccolaR$ denotes the subset of particles in $\bpiccola$ which belong to the ball $B(0,R)$. 
Therefore, the corresponding dimension $m=m(\bpiccola,R)$ is equal to the finite number $\# \bpiccolaR$ of particles.

For any continuous function $f\colon [0,1]\to\R^d$, let $w$ be its \emph{modulus of continuity}, that is, for any $\delta>0$, $w(f,\delta)\defeq\sup\{|f(t)-f(s)|, \,|t-s|<\delta\}$.  We say that a continuous path is \emph{nice} if it stays away from the origin, or if its modulus of continuity $w$ is bounded. More precisely, for any $\alpha,\delta,\eps >0$, we define the set of \emph{$(\alpha,\delta,\eps)$-nice paths} as
\begin{equation}\label{DefNicepath}
\NicePath(\alpha,\delta,\eps)
    \defeq \left\{ f\colon [0,1]\to\R^d \text{ continuous s.t.~} 
               \min_{s\in[0,1]}|f(s)| > \alpha \text{  or } w(f,\delta)\leq \eps \right\}. 
\end{equation}
Finally, for $\bx \in \D$, we define an event $\Omega_\bx\subset\Omega$, on which we will be able to construct the solution of the infinite dimensional two-type dynamics \eqref{eq:Sninfty}, as follows: 

\begin{equation}\label{eq:Omegax}
\begin{aligned}
    \Omega_\bx \defeq  \Big\{ \omega\in\Omega :\, ~ &
        \forall\Rgrande\in\N^*\, ~ \exists \Rpiccola\in\N^*\, ~ \forall R\ge \Rpiccola ,\ \\ 
    &\forall i\le n,\, ~  \Grande_i^{\bx,R}(\omega)\in\NicePath(\Rgrande+\sqrt{R},\frac{1}{\sqrt{R}},1), \\ 
    &\forall k\le m(\bpiccola,R),\, ~ \Piccola_k^{\bx,R}(\omega)\in\NicePath(\Rgrande+\rdep+\sqrt{R},\frac{1}{\sqrt{R}},1), \\
    &\forall i\le n,\, ~ \Grande_i^{\bx,R+1}(\omega)\in\NicePath(\Rgrande+\sqrt{R},\frac{1}{\sqrt{R}},1), \\
    &\forall k\le m(\bpiccola,R+1),\, ~ \Piccola_k^{\bx,R+1}(\omega)\in\NicePath(\Rgrande+\rdep+\sqrt{R},\frac{1}{\sqrt{R}},1) \Big\}.
\end{aligned}
\end{equation}
\smallbreak

\noindent
{\bf Step 3:} \emph{Convergence on $\Omega_\bx\subset \Omega$ of the approximating processes.}
\smallbreak

In the proposition below, we show that, for any fixed admissible initial configuration $\bx$ and $\omega\in\Omega_\bx$ (which fixes the Brownian paths), the trajectories of a hard sphere (or a particle) following the dynamics (\ref{eq:SnmR}$)_R$ are the same as soon as $R$ is large enough. This is due to the fact that the confining function $\psiPR$ vanishes on an increasingly large area $B(0,R)$ around the origin. Therefore, the path sequence indexed by $R$ converges.
Note that this convergence holds separately for each sphere and each particle, not for the process as a whole.

\begin{proposition}\label{PropStationnarite}
Fix $\bx$ in $\D$. The sequence of paths $\big(X^{\bx,R}(\omega),\omega\in\Omega_\bx\big)_{R \in \N^*}$ defined in \eqref{eq:XxR} converges to a 
limit process denoted by
$\disp X^{\bx} \defeq \left( \Grande_i^{\bx}, \Piccola_k^{\bx}, L_{ij}^{\bx}, \ell_{ik}^{\bx} \right)_{1\le i,j \le n,\, k\geq 1}$. 
This process is solution on $\Omega_\bx$ of the infinite-dimensional equation \eqref{eq:Sninfty} with initial configuration $\bx$.
\end{proposition}
\begin{proof} 
Fix $\bx=\bgrande\bpiccola \in \D$ and $\omega\in\Omega_\bx$. 

We first prove that, for fixed $1\leq i,j \leq n $ and $k\in\N^*$, the four path  sequences 
\begin{equation*}
    \big( \Grande_i^{\bx,R}(\omega) \big)_R\, ,\, \big( \Piccola_k^{\bx,R}(\omega)\big)_R\, ,\, \big( L_{ij}^{\bx,R}(\omega) \big)_R\, ,\, \big( \ell_{ik}^{\bx,R}(\omega) \big)_R
\end{equation*}
are eventually constant. 
Choose $\Rgrande\defeq\lceil \max\{ |\grande_i|, 1\le i \le n \}\rceil$ , where $\lceil z\rceil$ denotes the smallest integer larger than $z\in\R$.
The initial position $\grande_i$ of the $i$-th hard-sphere centre belongs therefore to $B(0,\Rgrande)$.

Since both paths $\Grande_i^{\bx,R}(\omega)$ and $\Grande_i^{\bx,R+1}(\omega)$ belong to the set $\NicePath(\Rgrande+\sqrt{R},1/\sqrt{R},1)$, then
\begin{equation*}
    w\big(\Grande_i^{\bx,R}(\omega),1/\sqrt{R}\big) \le 1 \text{ as soon as } \min_{t\in [0,1]}|\Grande_i^{\bx,R}(\omega,t)| \le \Rgrande+\sqrt{R},
\end{equation*}
and the same holds for $\Grande_i^{\bx,R+1}(\omega)$.
   
Moreover, since a path with $\delta$-modulus of continuity bounded by $\eps$ started in $B(0,\alpha)$ remains in $B(0,\alpha+ \frac{\eps}{\delta})$ for the time interval $[0,1]$ then, taking $\alpha = \Rgrande, \delta = 1/\sqrt{R} $ and $\eps =1$, then
\begin{equation*}
    \max_{t \in [0,1]}|\Grande_i^{\bx,R}(\omega,t)| \le \Rgrande+\sqrt{R} \text{ as soon as } \min_{t\in [0,1]}|\Grande_i^{\bx,R}(\omega,t)| \le \Rgrande .
\end{equation*}
The latter is verified by definition as
\begin{equation*}
    \min_{t\in [0,1]}|\Grande_i^{\bx,R}(\omega,t)| \le |\Grande_i^{\bx,R}(\omega,0)| \le \Rgrande.    
\end{equation*}
The same argument holds for $\Grande_i^{\bx,R+1}(\omega)$.

This implies in particular that every particle $\Piccola_k^{\bx,R}(\omega)$ which collides with some hard sphere $\Grande_i^{\bx,R}(\omega)$ belongs to $B(0,\Rgrande+\rdep + \sqrt{R})$ at the time of the collision. 

Since the path $\Piccola_k^{\bx,R}(\omega)$ belongs to $\NicePath(\Rgrande+\rdep+\sqrt{R},1/\sqrt{R},1)$, if it collides with some hard sphere, its $(1/\sqrt{R})$-modulus of continuity is bounded by $1$.
The same argument holds for  $\Piccola_k^{\bx,R+1}(\omega)$.
So, as particles that collide with a hard sphere at some time $t \in [0,1]$ cannot cover more than a distance $\sqrt{R}$ in the time interval $[0,1]$, their paths stay in the ball $B(0,\Rgrande+\rdep+2\sqrt{R})$.

As a consequence, for $R$ large enough in the sense that $R \geq \Rgrande+\rdep+2\sqrt{R}$ 
(this holds as soon as $R\geq \underline{R}\defeq4 + 2(\Rgrande + \rdep)$),
the hard spheres and the particles that visit  $B(0,\Rgrande+\rdep+\sqrt{R})$ stay in a region where the self-potentials $\psiPR$ and $\dot{\psi}^{R+1}$ vanish. 
That is, the $k$-th-particle dynamics computed at $\omega$ in \eqref{eq:SnmR} does not feel $\psiPR$ if it collides with hard spheres or if it starts in $B(0,\Rgrande+\rdep+\sqrt{R})$. 
Consequently, the sphere dynamics computed at $\omega$ in both equations \eqref{eq:SnmR} and $({\mathfrak S}_{m,R+1})$ coincide when $R \geq \underline{R}$, and the $k$-th-particle dynamics computed at $\omega$ coincide as soon as $\piccola_k \in B(0,\Rgrande+\rdep+\sqrt{R})$. 

The strong uniqueness in Theorem 2.2 of \cite{MultipleConstraint} allows us to deduce the existence of paths $\Grande_i^{\bx}(\omega)$,  $\Piccola_k^{\bx}(\omega)$, $L_{ij}^{\bx}(\omega)$, and $\ell_{ik}^{\bx}(\omega)$ such that
\begin{align*} 
    &\forall R \geq \underline{R}, \ \forall 1\le i,j \le n, \ \forall k \text{ such that }\piccola_k \in B(0,\Rgrande+\rdep+\sqrt{R}), \ \forall t\in[0,1], \\
    &\Grande_i^{\bx,R}(\omega,t) = \Grande_i^{\bx}(\omega,t), \, \Piccola_k^{\bx,R}(\omega,t) =  \Piccola_k^{\bx}(\omega,t), \, L_{ij}^{\bx,R}(\omega,t) = L_{ij}^{\bx}(\omega,t), \, \ell_{ik}^{\bx,R}(\omega,t) = \ell_{ik}^{\bx}(\omega,t).
\end{align*} 
By construction then, these paths satisfy the following SDE:
\begin{equation*}
\begin{cases}
\textrm{for } i, j \in \{1,\dots,n\}, \,   k \in \N^*,\,  t \in [0,1] ,                             \\ 
\disp \Grande_i^\bx(\omega,t) 
    = \grande_i +\sigmaG\, \WGrande_i(\omega,t) -\frac{1}{2} \sigmaG \int_0^t \nabla\psiG \big(\Grande_i^\bx(\omega,s)\big)\, d s \\ 
    \disp \phantom{\Grande_i^\bx(\omega,t) = \Grande_i +\sigmaG\, W_i(\omega,t) -}
    +\sum_{j=1}^n \int_0^t \big(\Grande_i^\bx(\omega,s)-\Grande_j^\bx(\omega,s)\big) \,d L_{ij}^\bx(\omega,s) \,\\
    \disp \phantom{\Grande_i^\bx(\omega,t) = \Grande_i +\sigmaG\, W_i(\omega,t) -}
    +\sum_{k=1}^{+\infty} \int_0^t \big(\Grande_i^\bx(\omega,s)-\Piccola_{k}^\bx(\omega,s)\big)\, d \ell_{ik}^\bx(\omega,s) , \\ 
    \disp \Piccola_k^\bx(\omega,t) = \piccola_k +\sigmaP\, \WPiccola_k(\omega,t) +\sigmaP^2 \sum_{i=1}^n \int_0^t \big(\Piccola_k^\bx(\omega,s) - \Grande_i^\bx(\omega,s)\big)\,  d \ell_{ki}^\bx(\omega,s) , \\ 
    \disp L_{ij}^\bx(\omega,0) = 0 , \quad L_{ij}^\bx \equiv L_{ji}^\bx, \quad L_{ij}^\bx(\omega,t) = \int_0^t \un_{|\Grande_i^\bx(\omega,s)-\Grande_j^\bx(\omega,s)|=2\rgrande} \, d L_{ij}^\bx(\omega,s), \quad L_{ii}^\bx \equiv 0 , \\ 
    \disp \ell_{ik}^\bx(\omega,0) = 0 , \quad \ell_{ik}^\bx \equiv \ell_{ki}^\bx, \quad \ell_{ik}^\bx(\omega,t) = \int_0^t \un_{|\Grande_i^\bx(\omega,s)-\Piccola_k^\bx(\omega,s)|=\rdep} \, d \ell_{ik}^\bx(\omega,s), \quad \ell_{ii}^\bx \equiv 0 .
\end{cases}
\end{equation*}
This concludes the proof.
\end{proof}
\smallskip
\noindent
{\bf Step 4:} \emph{The constructed solution of \eqref{eq:Sninfty} is defined on a full subset of $\Omega$.}

We first introduce a probability measure $\mu$ on the set of admissible configurations $\D$, as the law of $n$ hard spheres -- each one submitted to the self-potential $\psiG$ -- in an admissible Poisson bath of particles. 
In Section \ref{sec:reversible measure}  we will eventually prove that $\mu$ is indeed the reversible probability measure for the dynamics \eqref{eq:Sninfty}.

Let $\pi(d \bpiccola)$ denote the Poisson point process on $\MPiccola$ with intensity $\zpiccola>0$, where $\zpiccola$ is a fixed parameter (we let it vary only in Section~\ref{sec:packing}). We consider the probability measure $\mu$ on $\M$ with support in $\D$, defined by the following integral characterisation: for any positive measurable function $F$ on $\D$,
\begin{equation} \label{eq:mu}
    \int_\D F (\bx) \, \mu (d \bx)\defeq \frac{1}{Z}
\int_{\R^{nd}} \int_{\MPiccola} F(\bgrande\bpiccola)  \,  \un_{\D}(\bgrande\bpiccola) \, \pi (d \bpiccola) \,\otimes_{i=1}^n \lambda (d\grande_i).
\end{equation}
The normalisation constant
\begin{equation*}
    Z = \int_\M \un_{\D}(\bgrande\bpiccola)\, \pi (d \piccola)\, \otimes_{i=1}^n \lambda (d\grande_i) < + \infty
\end{equation*}
is finite since the measure $\lambda$ has finite mass.

Notice that, by considering $\otimes_{i=1}^n \lambda (d\grande_i)$, we have enforced an ordering on the hard spheres $\bgrande$. 
As such, one would expect an additional factor of $1/n!$, which is absorbed in the above normalisation constant $Z$.
%
\begin{proposition}
\label{PropNegligeabilite}
For $\mu$-a.e. $\bx \in \D$, 
\begin{equation*}
    P( \Omega_\bx ) =1.
\end{equation*}
Therefore, the limit process $\disp X^\bx$ constructed in Proposition \ref{PropStationnarite} is well-defined for $\mu$-almost every initial configuration $\bx$.
\end{proposition}

\begin{proof}
We aim to prove that $\disp \int_\D P\big( (\Omega_{\bx})^c \big) \, d \mu(\bx) =0 $. 

From the definition of $\Omega_\bx$ given in \eqref{eq:Omegax}, we can write its complement set $(\Omega_\bx)^c$ as
\begin{align*}
 \bigcup_{\Rgrande\in\N^*} \limsup_R
           \bigg\{
           &\exists i\le n: \Grande_i^{\bx,R} \in \BadPath \big(\Rgrande+\sqrt{R},\frac{1}{\sqrt{R}},1\big) \cup \BadPath \big(\Rgrande+\sqrt{R-1},\frac{1}{\sqrt{R-1}},1 \big) \text{ or }\\ 
           &\exists k\le m(\bpiccola,R): \Piccola_k^{\bx,R} \in \BadPath \big(\Rgrande+\rdep+\sqrt{R},\frac{1}{\sqrt{R}},1 \big) \cup \BadPath \big(\Rgrande+\sqrt{R-1},\frac{1}{\sqrt{R-1}},1\big)
           \bigg\}, 
\end{align*}
where $\BadPath$ denotes the set of \emph{bad paths}, complement of the set $\NicePath$ of nice paths defined in \eqref{DefNicepath}:
\begin{equation*}
    \BadPath(\alpha,\delta,\eps)
    \defeq \left\{ f\colon [0,1]\to\R^d \text{ continuous s.t. } 
               \min_{s\in[0,1]}|f(s)| \leq \alpha \text{  and } w(f,\delta)> \eps \right\} .    
\end{equation*}
In other words, a path in $\BadPath(\alpha,\delta,\eps)$ visits the ball $B(0,\alpha)$ during the time interval $[0,1]$ and its $\delta$-modulus of continuity is larger than $\eps$.

Since $\BadPath(\alpha,\delta,\eps)$ increases as $\alpha$ increases, and increases as $\delta$ increases, it suffices to prove that
$\disp \int_\D P( \tilde{\Omega}_\bx) \, d \mu(\bx) =0 $, where
\begin{align*}
\tilde{\Omega}_\bx \defeq \bigcup_{\Rgrande\in\N^*} \limsup_R
           & \left\{ 
             \exists i\le n : \Grande_i^{\bx,R} \in \BadPath(\Rgrande+\sqrt{R},\frac{1}{\sqrt{R-1}},1) \right.  \\ 
           & \left. \quad
             \text{ or } \exists k\le m(\bpiccola,R) : \Piccola_k^{\bx,R} \in \BadPath(\Rgrande+\rdep+\sqrt{R},\frac{1}{\sqrt{R-1}},1) \right\} \supset (\Omega_\bx)^c.
\end{align*}
Thanks to the Borel--Cantelli lemma, it is then sufficient to prove that, for any $\Rgrande\in\N^*$,
\begin{multline}
\label{ConvergenceSerieSurR}
\sum_{R \in \N^*} \int_\D P\Big( \exists i\le n,~ \Grande_i^{\bx,R} \in \BadPath(\Rgrande+\sqrt{R},\frac{1}{\sqrt{R-1}},1) \quad \text{ or }\\
                 \exists k\le m(\bpiccola,R),~ \Piccola_k^{\bx,R} \in \BadPath(\Rgrande+\rdep+\sqrt{R},\frac{1}{\sqrt{R-1}},1) 
                 \Big) \, d \mu(\bx) <+\infty .
\end{multline}
The convergence of the above series will derive from a precise control of the probability of bad paths under various reversible dynamics. This result is contained in Lemmas \ref{OscillationFiniteNumberBubbles} and \ref{OscillationInfiniteNumberBubbles} below. 
%
\begin{lemma} \label{OscillationFiniteNumberBubbles}
Let $\numR$ denote the probability measure on $\D_m$ obtained by normalising the measure $\nu^{m,R} $ defined  in \eqref{eq:numR}. 

The reversible solution $X^{m,R}$ of the SDE \eqref{eq:SnmR} with initial distribution $\numR$ satisfies the following inequality: for $0<\alpha<\alpha'$, $\delta\in (0,1)$ and $\eps>0$, 
\begin{equation*}
    P_\numR \left( \exists i\le n,\, \Grande_i^{m,R}\in\BadPath(\alpha,\delta,\eps) \textrm{ or } 
           \exists k\le m,\, \Piccola_k^{m,R}\in\BadPath(\alpha',\delta,\eps) \right)
    \le \frac{\mathfrak{c}_1}{\delta}\, \frac{n+m}{1-\e^{-\mathfrak{c}_0 \, \delta \alpha^2}}\, \e^{-\mathfrak{c}_0\,\eps^2/\delta},
\end{equation*}
where $\mathfrak{c}_0>0$ and $\mathfrak{c}_1>0$ are constants depending only on the dimension $d$ and the parameter $\sigmaP$.
\end{lemma}
For the sake of readability, the proof of this lemma is stated at the end of the section.
\medbreak

Consider now a Poissonian randomisation of the number $m$ of moving particles in the measure $\nu^{m,R}$. This leads to the definition of the following  probability measure on $\D$, mixture of $\nu^{m,R}$ measures:
\begin{equation} \label{eq:muR}
\mu^R \defeq \frac{1}{\ZRmixing} \sum_{m=0}^{+\infty} \frac{\zpiccola^m}{m!} \, \nu^{m,R} ,
\end{equation} 
where the normalisation constant $\ZRmixing$ is given by
\begin{equation*}
    \ZRmixing 
 =    \int_{\R^{dn}} \un_{\D}(\bgrande)\, 
                    \exp\left( \zpiccola \int_{\R^d\setminus\BGrande(\bgrande)} \e^{-\psiPR(x)} dx \right)\, \otimes_{i=1}^n \lambda (d\grande_i) .     
\end{equation*}
Recall from~\eqref{UnionBouees} that the interior of the set $\BGrande(\bgrande)\subset \R^d$ is the forbidden volume for the centres of particles around the configuration $\bgrande$.
\smallbreak

By arguments similar to those used in the proof of Lemma \ref{OscillationFiniteNumberBubbles}, we can also control the probability that the solution of \eqref{eq:SnmR} with initial distribution $\mu^R$ contains bad paths:
\begin{lemma}
\label{OscillationInfiniteNumberBubbles}
The solution of the SDE \eqref{eq:SnmR} with initial distribution $\mu^R$ satisfies the following inequality: for $0<\alpha<\alpha'$, $\delta\in (0,1)$ and $\eps>0$,
\begin{equation*}
\begin{split}
    P_{\mu^R}& \left( \exists i\le n: \Grande_i^{m,R}\in\BadPath(\alpha,\delta,\eps) \text{ or } 
           \exists k\le m: \Piccola_k^{m,R}\in\BadPath(\alpha',\delta,\eps) \right)\\
    &\le \frac{\mathfrak{c}_1}{\delta} \,
     \Big( n + \zpiccola \displaystyle \int_{\R^d} \e^{-\psiPR(x)} dx \Big) \,
    \frac{ \e^{-\mathfrak{c}_0\,\eps^2/\delta} }{1-\e^{-\mathfrak{c}_0 \, \delta \alpha^2}}\, ,
\end{split}
\end{equation*}
where $\mathfrak{c}_0>0$ and $\mathfrak{c}_1>0$ are universal constants depending only on the dimension $d$ and the parameter $\sigmaP$.
\end{lemma}
From Lemma \ref{OscillationInfiniteNumberBubbles}, together with Remark \ref{rem:psiPR}, it is now easy to obtain the convergence of the following series:
\begin{equation}\label{ConvergenceSerieMuRZ}
\begin{aligned}[t]
    &\sum_R  P_{\mu^R}\Big( \exists i\le n : \Grande_i^{m,R}\in\BadPath(\Rgrande+\sqrt{R},\frac{1}{\sqrt{R-1}},1) \text{ or } \\
    &\phantom{\sum_R  P_{\mu^R}\Big(}
    \exists k\le m : \Piccola_k^{m,R}\in\BadPath(\Rgrande+\rdep+\sqrt{R},\frac{1}{\sqrt{R-1}},1) \Big) \\
    &\le \displaystyle \sum_R \mathfrak{c}_1\sqrt{R-1} ~ \frac{n + \zpiccola \int_{\R^d} \e^{-\psiPR(x)} dx}{1-\exp(-\mathfrak{c}_0 (\Rgrande+\sqrt{R})^2 /\sqrt{R-1})}   e^{-\mathfrak{c}_0\sqrt{R-1}} < + \infty.
\end{aligned}    
\end{equation}
This is still not exactly the summability \eqref{ConvergenceSerieSurR} we are aiming for, but we are close. 

In order to conclude, we have to compare the two processes below, whose dynamics are given by the same SDE~\eqref{eq:SnmR}, ${m\in\N}$, but with different initial configuration distributions:
\begin{itemize}
    \item The process $X^{m,R}$, with random $m$ and initial configuration $X^{m,R}(0)$ chosen according to the probability measure $\mu^R$, introduced in \eqref{eq:muR};
    \item The process $X^{\bx,R}$, whose initial configuration $\bx$ is chosen according to the probability measure $\mu$. In particular, the random number $m=m(\bpiccola,R)$ corresponds to the number $\# \bpiccolaR$ of particles of $\bpiccola$ in $B(0,R)$, see \eqref{eq:XxR}.
\end{itemize}
Note that the first process is reversible -- as it is given by a mixture of reversible processes -- while the second one is not, since, e.g., its initial law only weighs configuration of particles concentrated in $B(0,R)$.

In order to estimate the difference between these two processes, we consider the total variation distance between their laws, denoted by $\mathbf{d}_{TV}(R)$. It is defined as usual as the supremum over all measurable sets $A$ of continuous paths with values in $\D$\,:
\begin{equation*}
\begin{split}
\mathbf{d}_{TV}(R) 
& \defeq\sup_{A\subset C([0,1],\D)} 
         \bigg|
         \int_\D P \big(X^{\# \bpiccola,R}  \in A | X^{\# \bpiccola,R}(0)=\bx  \big) \, \mu^R (d\bx) \\
         &\phantom{\defeq\sup_{A \subset C([0,1],\D)} 
         \bigg|
         \int_\D}
         - \int_\D P \big(X^{m(\bpiccola,R),R}  \in A | X^{m(\bpiccola,R),R}(0)=\bx  \big) \, \mu (d\bx) 
         \bigg|  \\
& =\sup_{A\subset C([0,1],\D)} 
        \bigg| 
        \int_\D F_A(\bx) ~\mu^R(d\bx) -\int_\D F_A(\bgrande\bpiccolaR) \, \mu(d\bx)
        \bigg|,
\end{split}
\end{equation*}
where, for any $m\geq 0$, the function $F_A$ is defined on $\D_m$ by $F_A(\bx) \defeq P\big(X^{m,R}  \in A | X^{m,R}(0)=\bx \big)$. 

In the second integral, we can disintegrate the Poisson point measure $\pi$ (which models the law of the particles under $\mu$) into the product of the Poisson point measure $\pi_{|_R}$ inside $B(0,R)$ and the Poisson point measure $\pi_{|_{R^c}}$ outside $ B(0,R)$.
Denoting
\begin{equation*}
Z_R \defeq \int_{\MPiccola} \int_{\R^{dn}} \un_{\D}(\bgrande\bpiccolay) ~ 
     \e^{\zpiccola |B(0,R)\setminus\BGrande(\bgrande)|} 
     \, \otimes_{i=1}^n \lambda(d\grande_i) \, d \pi_{|_{R^c}}(\bpiccolay),
\end{equation*}
we obtain that $\mathbf{d}_{TV}(R)$ is the supremum over $A$ of the following expression:
\begin{equation*}
\begin{split}
    &\bigg|\frac{1}{\ZRmixing} \int_{\R^{dn}} \un_{\D}(\bgrande) 
  \Big( F_A(\bgrande) \\
  &\phantom{\bigg|\frac{1}{\ZRmixing} \int_{\R^{dn}} \un_{\D}}
  +\sum_{m=1}^{+\infty} \frac{\zpiccola^m}{m!} 
                            \int_{\R^{dm}} F_A(\bgrande\piccola_1\dots \piccola_m) \prod_{k=1}^m \big( \un_{\piccola_k\notin\BGrande(\bgrande)} \, \e^{-\psiPR(\piccola_k)} \big)
                            \otimes_{k=1}^m d\piccola_k \Big)
    \otimes_{i=1}^n\lambda(d\grande_i)\\
    &\phantom{\bigg|}-\frac{1}{Z_R} \int_{\MPiccola} \int_{\R^{dn}} \un_{\D}(\bgrande\bpiccolay) 
   \Big( F_A(\bgrande) \\
    &\phantom{\frac{1}{Z_R} \int_{\MPiccola} \int_{\R^{dn}} }
          +\sum_{m=1}^{+\infty} \frac{\zpiccola^m}{m!} 
           \int_{B(0,R)^m} F_A(\bgrande\piccola_1\dots\piccola_m) \prod_{k=1}^m \un_{\piccola_k\notin\BGrande(\bgrande)} 
           \otimes_{k=1}^m d\piccola_k  \Big)
    \otimes_{i=1}^n \lambda(d\grande_i) \,\pi_{|_{R^c}}(d\bpiccolay)\bigg|\\
    &\leq \int_{\MPiccola}
    \bigg|\int_{\R^{dn}} F_A(\bgrande) \Big( \frac{\un_{\D}(\bgrande)}{\ZRmixing}
                                            -\frac{\un_{\D}(\bgrande\bpiccolay)}{Z_R} \Big) 
      \\
    &\phantom{\le \int_{\MPiccola}\bigg|\int_{\R^{dn}}}
      +\sum_{m=1}^{+\infty} \frac{\zpiccola^m}{m!}                          
                    \int_{\R^{dm}} F_A(\bgrande\bpiccola) 
                                   \prod_{k=1}^m \Big( \un_{\piccola_k\notin\BGrande(\bgrande)}\, \e^{-\psiPR(\piccola_k)} \Big)\\
    &\phantom{\le \int_{\MPiccola}\bigg|\int_{\R^{dn}}+\sum_{m=1}^{+\infty}}
    \Big( \frac{\un_{\D}(\bgrande)}{\ZRmixing}  
                                          -\frac{\un_{\D}(\bgrande\bpiccolay)}{Z_R}
                                           \un_{\bpiccola \subset B(0,R)} \Big)
                    \, \otimes_{k=1}^m d\piccola_k  
  \, \otimes_{i=1}^n \mesGrande(d\grande_i)             
  \bigg|   
  \,\pi_{|_{R^c}}(d\bpiccolay).
\end{split}
\end{equation*}
We then have that
\begin{equation*}
\begin{split}
    \mathbf{d}_{TV}(R)
     &\le \left| \frac{\ZRmixing}{Z_R}-1 \right| + n \left( 1+\frac{\ZRmixing}{Z_R} \right) \frac{\e^{\zpiccola |B(0,R)|}}{\ZRmixing}
      \, \e^{ \zpiccola \int_{B(0,R)^c} \e^{-\psiPR(x)} \, d x }
      \int_{B(0,R-\rdep)^c} \, \e^{-\psiG(x)} \, d x \\
  &\phantom{\le \frac{\ZRmixing}{Z_R} + n \left( 1+\frac{\ZRmixing}{Z_R} \right) \frac{\e^{\zpiccola |B(0,R)|}}{\ZRmixing}}
+\zpiccola\, \frac{\ZRmixing}{Z_R} \int_{B(0,R)^c} \e^{-\psiPR(x)} \, d x,
\end{split}
\end{equation*}
where the last inequality follows from the fact that the function $F_A$ is bounded by one, and by carefully reordering and upperbounding each term.

We now need the following fine estimate on the asymptotic behaviour for large $R$ of the two normalisation constants $\ZRmixing$ and $Z_R$:
\begin{equation*}
   1 \ \le \ \frac{\ZRmixing}{Z_R} \ \le \ \displaystyle \exp \Big( \zpiccola \displaystyle \int_{B(0,R)^c} \e^{-\psiPR(x)} dx \Big) , 
\end{equation*}
where the upper bound converges quickly to 1, since the exponent is summable in $R$, as stated in \eqref{eq:psiPR}. 
Therefore, there exists a positive constant $\mathfrak{c}_3$ depending only on $n, \psiG,\zpiccola$ such that, for $R$ large enough,
\begin{equation}\label{DistanceVariationSommable}
\begin{aligned}
    \mathbf{d}_{TV}(R) 
\ &\le \ \mathfrak{c}_3 \Big( \int_{B(0,R-\rdep)^c} \, \e^{-\psiG(x)} \, d x + \int_{B(0,R)^c} \e^{-\psiPR(x)} dx \Big) \\
&\le \ \mathfrak{c}_3 \left( \frac{ \int_{\R^d} |x|^2 \, \e^{-\psiG(x)} \, d x }{ (R-\rdep)^2 } 
                              + \int_{B(0,R)^c} \e^{-\psiPR(x)} dx \right).
\end{aligned}
\end{equation}
Thanks to assumptions \eqref{eq:psiG} and \eqref{eq:psiPR} on $\psiG$ and $\psiPR$, respectively, the right-hand side is summable.

The convergence of the series in \eqref{ConvergenceSerieMuRZ} implies the summability \eqref{ConvergenceSerieSurR} for every $\Rgrande$,
which in turn implies that $\Omega_\bx$ is $\mu$-a.s.~a set of full measure.
This completes the proof of Proposition \ref{PropNegligeabilite}.
\end{proof}
We are then only left with proving Lemma \ref{OscillationFiniteNumberBubbles}.
\begin{proof}[Proof of Lemma \ref{OscillationFiniteNumberBubbles}]
According to \eqref{eq:SnmR}, for any $i \in \{1,\dots,n\}$, the process
\begin{equation}\label{eq:WiBM}
\begin{split}
&\sigmaG \WGrandei (t) 
 =  \Grande_i^{m,R} (t) - \Grande_i^{m,R} (0) 
    + \frac{1}{2} \int_0^t \sigmaG \ \nabla\psiG \big(\Grande_i^{m,R} (s)\big)\, d s   \\ 
& \qquad  - \sum_{j=1}^n \int_0^t \big(\Grande_i^{m,R}-\Grande_j^{m,R})(s)\big)\, dL_{ij}(s)  
       - \sum_{k=1}^{m} \int_0^t (\Grande_i^{m,R}-\Piccola_{k}^{m,R})(s)\, d \ell_{ik}(s)
\end{split}
\end{equation} 
is a Brownian motion.
Since $X^{m,R}$ is time-reversible, the process
$\left( \Grande_i^{m,R}(t) \right)_{t\in[0,1]}$ has the same distribution as the backward process $\left( \Grande_i^{m,R}(1-t) \right)_{t\in[0,1]} $.
Consequently the process $\underleftarrow{\WGrandei}$ 
obtained by replacing $\Grande_i^{m,R} (\cdot)$ by $\Grande_i^{m,R} (1-\cdot)$ in \eqref{eq:WiBM} 
is also a Brownian motion. Moreover, it satisfies
\begin{align*}
\sigmaG\, \underleftarrow{\WGrandei}(t)
& = \Grande_i^{m,R} (1-t) - \Grande_i^{m,R} (1) 
    + \frac{1}{2} \int_{1-t}^1 \sigmaG \ \nabla\psiG \big(\Grande_i^{m,R} (s)\big)\, d s   \\ 
&~~~   - \sum_{j=1}^n \int_{1-t}^1 \big(\Grande_i^{m,R}-\Grande_j^{m,R})(s)\big)\,d L_{ij}(s)    
       - \sum_{k=1}^{m} \int_{1-t}^1 (\Grande_i^{m,R}-\Piccola_{k}^{m,R})(s)\, d \ell_{ik}(s). 
\end{align*} 
Summing \eqref{eq:WiBM} with the above expression at time $1-t$,  we get
\begin{align*}
\sigmaG\, \WGrande_i(t) + \sigmaG\, \underleftarrow{\WGrande}_i(1-t)
& = 2\Grande_i^{m,R} (t) - \Grande_i^{m,R} (0) - \Grande_i^{m,R} (1) 
    + \frac{1}{2} \int_0^1 \sigmaG \ \nabla\psiG \big(\Grande_i^{m,R} (s)\big)\, d s   \\ 
&~~~   - \sum_{j=1}^n \int_0^1 \big(\Grande_i^{m,R}-\Grande_j^{m,R})(s\big))dL_{ij}(s)    
       - \sum_{k=1}^{m} \int_0^1 (\Grande_i^{m,R}-\Piccola_{k}^{m,R})(s) d\ell_{ik}(s),
\end{align*}   
where the sum of the three integral terms is equal to 
$\sigmaG\, \underleftarrow{\WGrande}_i(1) - \Grande_i^{m,R}(0) + \Grande_i^{m,R}(1)$.\\
Therefore, one can express $\Grande_i^{m,R(\cdot)}$  without local-time terms as follows:
\begin{equation}
    \label{eq:Xisanstempslocal}
\Grande_i^{m,R}(t) = \Grande_i^{m,R}(0) + 
\frac{1}{2} \big( \WGrandei(t) + \underleftarrow{\WGrandei}(1-t) - \underleftarrow{\WGrandei}(1) \big), \ t\in [0,1].
\end{equation} 
Similarly,
 \begin{equation*}
    \label{eq:Xksanstempslocal}
    \Piccola_k^{m,R} (t) = \Piccola_k^{m,R} (0) 
   + \frac{1}{2} \big(\sigmaP\, \WPiccolak (t) + \sigmaP\, \underleftarrow{\WPiccolak}(1-t) -\sigmaP\, \underleftarrow{\WPiccolak}(1) \big), \ t \in [0,1].
\end{equation*}
The $n+m$ components of $X^{m,R}$ are exchangeable, thus
\begin{multline*}
   P_\numR \left( \exists i\le n,~ \Grande_i^{m,R}\in\BadPath(\alpha,\delta,\eps) \text{ or } 
           \exists k\le m,~ \Piccola_k^{m,R}\in\BadPath(\alpha',\delta,\eps) \right)    \\
   \le n~P_\numR \left( \Grande_1^{m,R}\in\BadPath(\alpha,\delta,\eps) \right) 
       + m~P_\numR \left( \Piccola_1^{m,R}\in\BadPath(\alpha',\delta,\eps) \right) .
\end{multline*} 
Summing over all possible initial positions, we get
\begin{align*}
  & P_\numR \left( \Grande_1^{m,R}\in\BadPath(\alpha,\delta,\eps) \right) \\
   & = \sum_{k=0}^{+\infty} 
   P_\numR \left( k\alpha \le |\Grande_1^{m,R}(0)| < (k+1)\alpha \text{  and } 
                            \inf_{[0;1]}|\Grande_1^{m,R}| \le \alpha  \text{  and } w(\Grande_1^{m,R},\delta)> \eps \right) .
\end{align*}
A path which starts outside $B(0,k\alpha)$ and visits $B(0,\alpha)$ before time $1$ necessarily has an oscillation larger than $(k-1)\alpha\delta$ in some time interval of length $\delta$; moreover, for $k$ large enough, $(k-1)\alpha\delta$ is larger than $\eps$. Therefore, the first two conditions in the above event imply the third one as soon as $\displaystyle k>1+\frac{\epsilon}{\alpha\delta}$. Hence,
\begin{align*}
    P_\numR \left( \Grande_1^{m,R}\in\BadPath(\alpha,\delta,\eps) \right)    
    &\le P_\numR \left( |\Grande_1^{m,R}(0)| < \alpha +\epsilon/\delta \text{  and } 
                   w(\Grande_1^{m,R},\delta) > \eps \right) \\
    & ~~~~~~ + \sum_{k>1+\eps/\alpha\delta} P_\numR \Big( k\alpha \le |\Grande_1^{m,R}(0)| < (k+1)\alpha  \\[-.8em]
    & \phantom{~~~~~~ + \sum_{k>1+\eps/\alpha\delta} P_\numR \Big(}
        \ \text{ and } w(\Grande_1^{m,R},\delta) > (k-1)\alpha\delta \Big) .
\end{align*}
Thanks to the decomposition \eqref{eq:Xisanstempslocal}, 
\begin{equation*}
    w(\Grande_1^{m,R},\delta) > \eps \quad \Rightarrow \quad w(\WGrandeone,\delta) > \eps \ 
\textrm{ or } \, w(\underleftarrow{\WGrandeone},\delta) > \eps.    
\end{equation*}
Therefore, we have
\begin{equation*}\label{InegModuleContinuiteBrownien}
\begin{split}
    & P_\numR \Big( \Grande_1^{m,R}\in\BadPath(\alpha,\delta,\eps) \Big)   \\ 
    & \le P_\numR \big( |\Grande_1^{m,R}(0)| < \alpha+\frac{\eps}{\delta} \big)  
             \Big( P_\numR \big( w(\WGrandeone,\delta) > \eps \big)
                     + P_\numR \big( w(\underleftarrow{\WGrandeone},\delta) > \eps \big) \Big)  \\
    & ~~~~~~ + \sum_{k>1+\eps/\alpha\delta} 
              \Big( P_\numR \big( w(\WGrandeone,\delta) >(k-1)\alpha\delta \big)
                    + P_\numR \big( w(\underleftarrow{\WGrandeone},\delta)  >(k-1)\alpha\delta \big) \Big).
\end{split}
\end{equation*}
Note now the following standard estimate on the modulus of continuity of any $d$-dimensional Brownian motion $W$: there exist two universal constants $\mathfrak{c}_1>0$ and $\mathfrak{c}_2>0$, depending only on the dimension $d$, such that
\begin{equation*}
    P\left( w(W,\delta) \ge \eps \right)\, \le\, \frac{\mathfrak{c}_1}{4\delta} ~ \exp(-\frac{\mathfrak{c}_2\eps^2}{\delta}).   
\end{equation*}
We then have
\begin{equation*}
\begin{split}
    & P_\numR \Big( \Grande_1^{m,R}\in\BadPath(\alpha,\delta,\eps) \Big)\\
    &\le \frac{\mathfrak{c}_1}{2\delta} ~ P_\numR \Big( |\Grande_1^{m,R}(0)| < \alpha+\frac{\eps}{\delta} \Big) \, 
                                  \exp(-\frac{\mathfrak{c}_2\eps^2}{\delta}) + \frac{\mathfrak{c}_1}{2\delta} \, \sum_{k>1+\eps/\alpha\delta} \exp(-\mathfrak{c}_2(k-1)^2\alpha^2\delta)    \\
    &  \le \frac{\mathfrak{c}_1}{2\delta} \,
             \Big( P_\numR \big( |\Grande_1^{m,R}(0)| < \alpha+\frac{\eps}{\delta} \big) 
                    + \sum_{j=0}^{+\infty} \exp(-\mathfrak{c}_2 j\alpha^2\delta) \Big) 
             \exp\big(-\frac{\mathfrak{c}_2\eps^2}{\delta}\big) \\
    &  \le \frac{\mathfrak{c}_1}{\delta} ~ \frac{1}{1 -\exp(-\mathfrak{c}_2\alpha^2\delta)} ~ 
               \exp\big(-\frac{\mathfrak{c}_2\eps^2}{\delta}\big), 
\end{split}
\end{equation*}
By similar arguments, we obtain the following  upper bound for the particles:
\begin{equation*}
    P_\numR \left( \Piccola_1^{m,R}\in\BadPath(\alpha',\delta,\eps) \right)
 \le \frac{\mathfrak{c}_1}{\delta} ~ \frac{1}{1 -\exp(-\mathfrak{c}_2(\alpha')^2\delta/\sigmaP^2)} ~ 
     \exp(-\frac{\mathfrak{c}_2\eps^2}{\sigmaP^2\delta}). 
\end{equation*}
Choosing $\mathfrak{c}_0\defeq \displaystyle \frac{\mathfrak{c}_2}{\max(1,\sigmaP^2)}$ yields 
the claimed estimate.
\end{proof} 

\medskip
\subsection{A (two-type) reversible measure} \label{sec:reversible measure}
The identification  of reversible measures associated to the dynamics \eqref{eq:Sninfty} is mathematically and physically relevant. We address it in this section. 

\begin{proposition}\label{PropReversibility}
Consider the solution of the two-type infinite-di\-men\-sional equation \eqref{eq:Sninfty}, whose initial condition $X(0)$ is random  and distributed according to the probability measure $\mu$ defined on $\M$ by  \eqref{eq:mu}. This solution is time-reversible. 
\end{proposition}
\begin{proof}
Recall that, thanks to Propositions \ref{PropStationnarite} and \ref{PropNegligeabilite}, for any fixed admissible initial condition $\bx\in\D$, we can construct a path $X^{\bx}$ solving \eqref{eq:Sninfty} on the full set $\Omega_\bx$. 

We aim to prove that the process $X^{\boldsymbol{X}_0}$ is time-reversible as soon as $\boldsymbol{X}_0$ is a $\mu$-distributed point process independent of the Brownian motions  $(\WGrande_i)_i$'s and $(\WPiccola_k)_k$'s.

The process $X^{\boldsymbol{X}_0}$ is time-reversible on $[0,1]$ if, for any time $T$ in $[0,1]$, the backward process $\big(X^{\boldsymbol{X}_0}(T-t)\big)_{0\le t\le T}$ has the same distribution as the forward process $\big(X^{\boldsymbol{X}_0}(t)\big)_{0\le t\le T}$. 
Equivalently, one has to prove that, for any times $0\leq t_1 < \dots <  t_j \leq T, \, j\in\N$, and any bounded continuous local functions $F_1, \dots, F_j$ on $\M$,
\begin{equation*}
    \int_\D E\bigg( \prod_{i=1}^j F_i\big( X^{\bx}(T-t_i) \big) 
                   -\prod_{i=1}^j F_i\big( X^{\bx}(t_i) \big)  \bigg) \mu(\d\bx) = 0.   
\end{equation*}
Since $X^{\bx}$ was obtained in Proposition \ref{PropStationnarite} as limit of $(X^{\bx,R})_R$, 
it is sufficient to prove that
\begin{equation*}
    \lim_{R\to+\infty} \int_\D E\bigg( \prod_{i=1}^j F_i\big( X^{\bx,R}(T-t_i) \big) 
                                      -\prod_{i=1}^j F_i\big( X^{\bx,R}(t_i) \big)  \bigg) \mu(\d\bx) = 0.
\end{equation*}
Analogously to the proof of the existence of the two-type process, we split the computation of the integral term into two terms, using the measure $\mu^R$, see \eqref{eq:muR}, mixture of the measures $\nu^{m,R}$, themselves reversible under the finite-dimensional dynamics \eqref{eq:SnmR}. 
Therefore,
\begin{align*} 
&  \int_\D E\bigg( \prod_{i=1}^j F_i( X^{\bx,R}(T-t_i) ) 
                   -\prod_{i=1}^j F_i( X^{\bx,R}(t_i) ) \bigg) \mu(\d\bx) \\
& \le \underbrace{ \int_\D E\bigg( \prod_{i=1}^j F_i( X^{m(\bpiccola,R),R}(T-t_i) ) 
                                   -\prod_{i=1}^j F_i( X^{m(\bpiccola,R),R}(t_i) ) \bigg) \mu^R(\d\bx) }_{ \text{ vanishes due to reversibility
                                   } }
       + \prod_{i=1}^j \|F_j\|_\infty ~ \mathbf{d}_{TV}(R).
\end{align*} 
We know from \eqref{DistanceVariationSommable} that the distance in total variation $\mathbf{d}_{TV}(R)$ is summable in $R$, thus \emph{a fortiori} it tends to zero.
This completes the proof of the time-reversibility of the solution of \eqref{eq:Sninfty} when the initial condition is $\mu$-distributed.
\end{proof}

\section{Occurrence of a new depletion interaction between hard spheres} \label{sec:depletion}

We are now interested in the projection of the two-type model considered till now onto the system of just the $n$ hard spheres. Our attention will first focus on the projection $\mugrande$ of the reversible measure $\mu$ studied in Section \ref{sec:reversible measure}. A new interaction between the hard  spheres induced by the small ones will be observed as a {\em depletion interaction}. Its specific properties are pointed out, in particular the fact that it is highly local, see Figure \ref{fig:graphd2} and Figure \ref{fig:graphd3}. We then identify an $n$-dimensional random gradient dynamics whose reversible measure is given by $\mugrande$. Finally, we prove that, asymptotically as the density $\zpiccola$ of the (hidden) particles tends to infinity, the measure $\mugrande$ concentrates around remarkable geometrical configurations: they form sphere clusters that maximise their contact number, which is part of an important -- and difficult -- topic in discrete geometry.

\subsection{The projection of the two-type reversible measure: the occurrence of a depletion interaction}\label{sec:trace}

We study here the projection of the reversible measure $\mu$ onto the $n$ hard spheres.
\smallbreak

We first mention that in the case $n=+\infty$, S. Jansen and D. Tsagkarogiannis computed in~\cite{JT19,JT20} the projection of the {\em partition function} of a two-type grand canonical Gibbs process on the system of hard spheres. They prove the emergence of an additional induced interaction between the spheres, identified as a {\em depletion} interaction. This kind of interaction was already known and studied in the physical literature, see, e.g.,~\cite{RED00} and~\cite{LT11}. 
We adapt their computations to our setting, where the number of hard spheres is fixed to a finite $n$.
\smallbreak

Recall that we defined in Section \ref{sec:mathmodel} a slightly enlarged version of the hard spheres, 
by placing around them a spherical shell with size $\rpiccola$, called depletion shell. The new feature is that these enlarged spheres with radius $\rdep= \rgrande + \rpiccola$ may overlap pairwise, and this will be at the origin of the depletion interaction.

\begin{proposition} \label{prop:tracemesureequ}
Consider $\mu$, the reversible probability measure on $\M$ of the two-type system defined by \eqref{eq:mu}. Its projection onto the $n$ hard sphere system is a probability measure $\mugrande_{\zpiccola}$ on $\MGrande$ defined as
\begin{equation} \label{eq:muronde}
    \mugrande_{\zpiccola} (d\bgrande) = \frac{1}{Z_{\zpiccola}} \exp \big(- \zpiccola \, \nrg (\bgrande) \big) \ \un_{\D} (\bgrande) \,  \otimes_{i=1}^n \lambda (d\grande_i),
\end{equation}
where $\nrg (\bgrande)$, the energy of the configuration $\bgrande \in \M$, is given by
\[    \nrg (\bgrande) \defeq \displaystyle \vol{\BGrande(\bgrande)} = \displaystyle \vol{\bigcup_{i=1}^{n}B(\grande_i,\rdep) } .\]
\end{proposition}

\begin{proof}
Integrating the measure $\mu$ over bounded test functions $F$ supported on $\MGrande$ yield
\begin{equation*}
\begin{split}
    \int_{\MGrande}F(\bgrande)\, \mugrande_{\zpiccola} (d \bgrande) 
    &\defeq \int_{\M}F(\bgrande) \, \mu(d \bx)  = \frac{1}{Z_{\zpiccola}}\int_{\R^{dn}} F(\bgrande) \, \int_{\MPiccola} \un_{\D}(\bgrande\bpiccola)\, \pi (d \bpiccola)
    \ \otimes_{i=1}^n \lambda(d\grande_i)\\
    &=\frac{1}{Z_{\zpiccola}}\int_{\R^{dn}} F(\bgrande) \Big(\int_{\MPiccola} \un_{\D}(\bgrande\bpiccola_{|_{R(\bgrande)}}) 
    \,
    \pi_{|_{R(\bgrande)}}
    (d \bpiccola)\ \otimes \pi_{|_{R(\bgrande)^c}}
    (d \bpiccola)\Big) \un_{\D} (\bgrande)\, \otimes_{i=1}^n \lambda(d\grande_i),
\end{split}
\end{equation*}
where $R(\bgrande)$ is a radius large enough for $B\big(0,R(\bgrande)\big)$ to contain $\BGrande(\bgrande)$. Since the particles outside of $B\big(0,R(\bgrande)\big)$ are independent of the ones inside, we get
\begin{equation*}
\begin{split}
    \int_{\MGrande}F(\bgrande)\, \mugrande_{\zpiccola}(d \bgrande) 
    &=\frac{1}{Z_{\zpiccola}}\int_{\R^{dn}} F(\bgrande) \int_{\MPiccola} 
    \un_{\BGrande(\bgrande)^c} (\bpiccola) \,
    \pi_{|_{R(\bgrande)}}(d \bpiccola)\ \un_{\D} (\bgrande)\, \otimes_{i=1}^n \lambda(d\grande_i) \\
    &= \frac{1}{Z_{\zpiccola}}\int_{\R^{dn}} F(\bgrande) \exp\big(-\zpiccola\, \vol{\BGrande(\bgrande)}\big)  \un_{\D} (\bgrande)\, \otimes_{i=1}^n \lambda(d\grande_i),
\end{split}
\end{equation*}
which is equivalent to the claim.
\end{proof}
\begin{remark}
By the inclusion-exclusion rule, 
\begin{equation}\label{eq:inex}
     \nrg (\bgrande) = \sum_{k=1}^n \, \vol{B(\grande_k,\rdep)} + \sum_{k=2}^n\sum_{1\leq i_1<\dots< i_k\leq n} \phidep_k(\grande_{i_1},\dots,\grande_{i_k}),
\end{equation}
where the function $\phidep_k,\, k\geq 2$, is a symmetric translation-invariant function on $(\R^d)^k$ given by 
\begin{equation}\label{eq:kbodydepletion}
    \phidep_k(x_{1},\dots,x_{k}) : = (-1)^{k-1} \,\vol{B(x_{1},\rdep)\cap\dots\cap B(x_{k},\rdep)}.
\end{equation}
The function $\phidep_k$ is called the $k$-body \emph{depletion interaction}. It is highly dependent on the proportionality factor $\rho$ between the particle radius and the hard sphere radius.

Putting everything together yields
\begin{equation*}
    \mugrande_{\zpiccola}(d \bgrande) 
    = \exp \bigg( -\zpiccola \, \Big( n  v_d \,\rdep\rspace ^d + \sum_{k=2}^n\sum_{1\leq i_1<\dots< i_k\leq n} \phidep_k(\grande_{i_1},\dots,\grande_{i_k}) \Big) \bigg) 
    \un_{\D} (\bgrande) \, \otimes_{i=1}^n \lambda (d\grande_i),
\end{equation*}
where $v_d$ denotes the volume in $\R^{d}$ of the unit sphere.
\end{remark}

In the following lemma, we give thresholds on $\rho$ for three-body or $k$-body interactions to occur. The computation of $\rho_2$ is classical, see for example~\cite{HL02,LT11}; see also~\cite{WJL22} for a geometric computation in the case of differently-shaped bodies.

\begin{lemma} \label{lem:rho2rho3}
    The multi-body depletion interaction reduces to a pair interaction as soon as the size proportionality factor $\rho$ between the  particles and the hard spheres is bounded from above by $\rho_2\defeq \frac{2}{3}\sqrt{3}-1\simeq 0.1547$. Moreover, $4$-body depletion interactions can appear, in dimension $d=2$, only if $\rho > \rho_3\defeq \sqrt{2}-1 \simeq 0.4142$, in dimension $d\geq 3$, only if $\rho > \rho_3\defeq \sqrt{3/2}-1 \simeq 0.2247$. 
\end{lemma}
\begin{figure}[ht]
\begin{center}
\begin{subfigure}[b]{0.2\textwidth}
    \includegraphics[height=\textwidth]{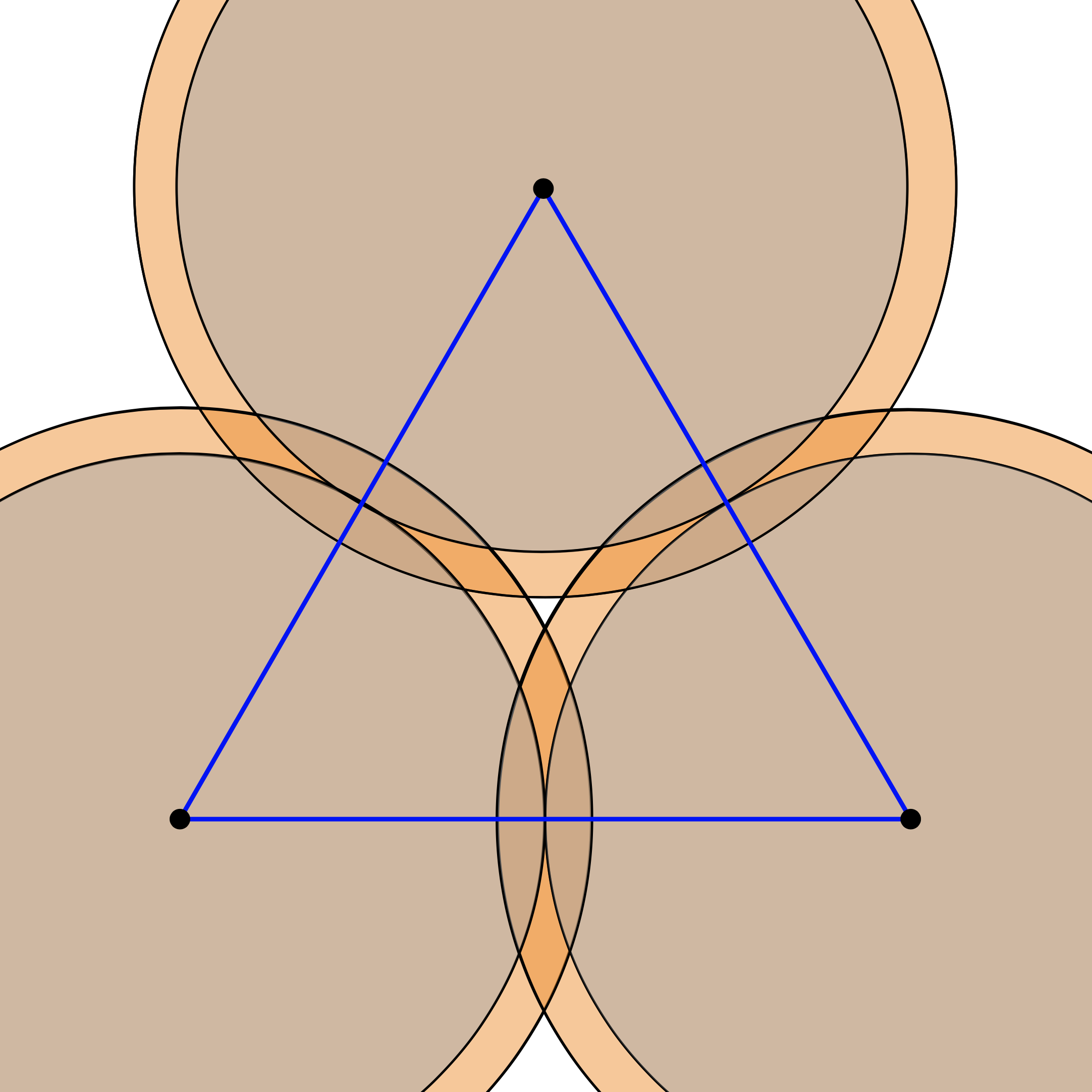}\hspace{1em}
    \caption{ }
    \label{subfig-1:dummy}
\end{subfigure}\hspace{2em}
\begin{subfigure}[b]{0.2\textwidth}
    \includegraphics[height=\textwidth]{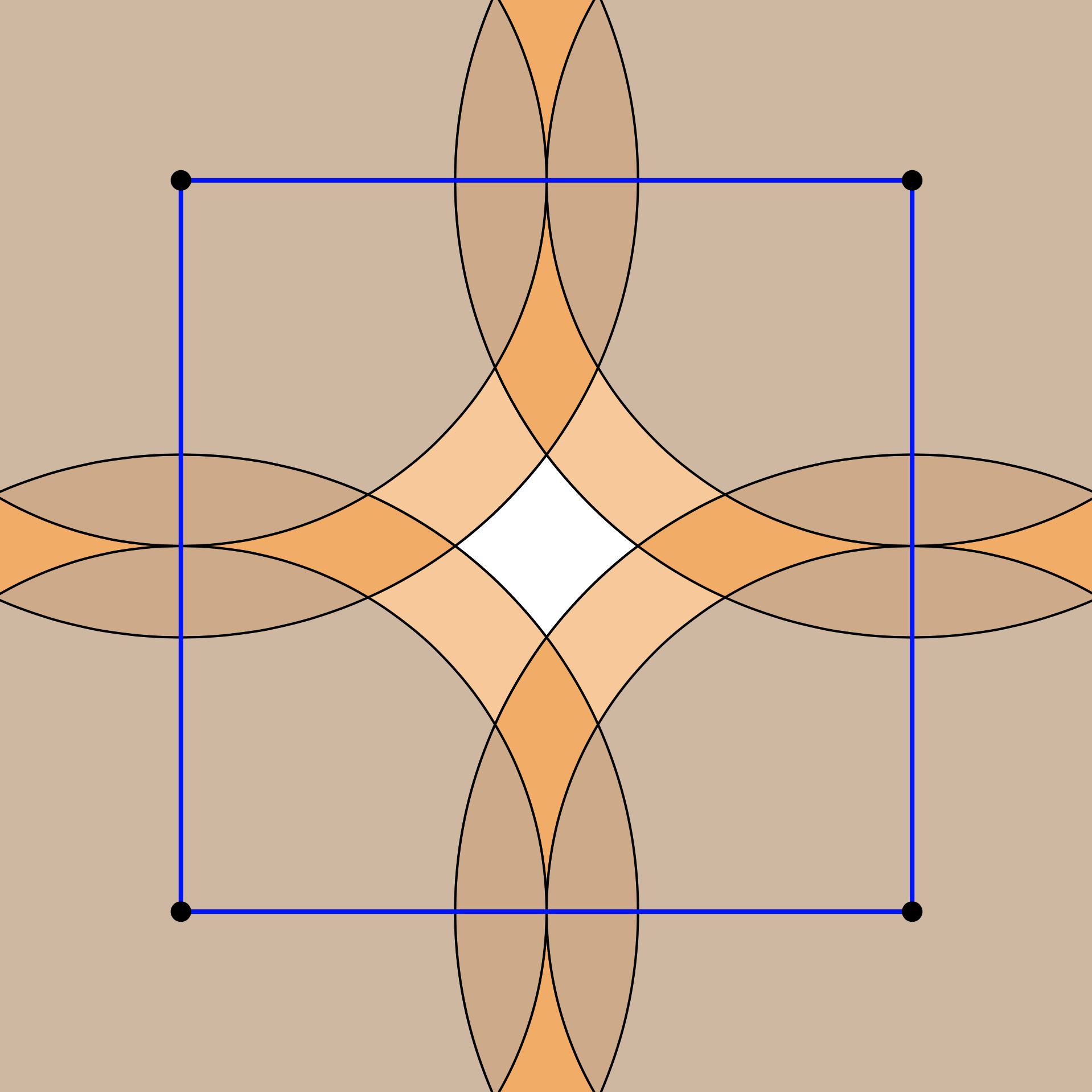}\hspace{1em}
    \caption{ }
    \label{subfig-2:dummy}
\end{subfigure}
\begin{subfigure}[b]{0.2\textwidth}
    \includegraphics[height=\textwidth]{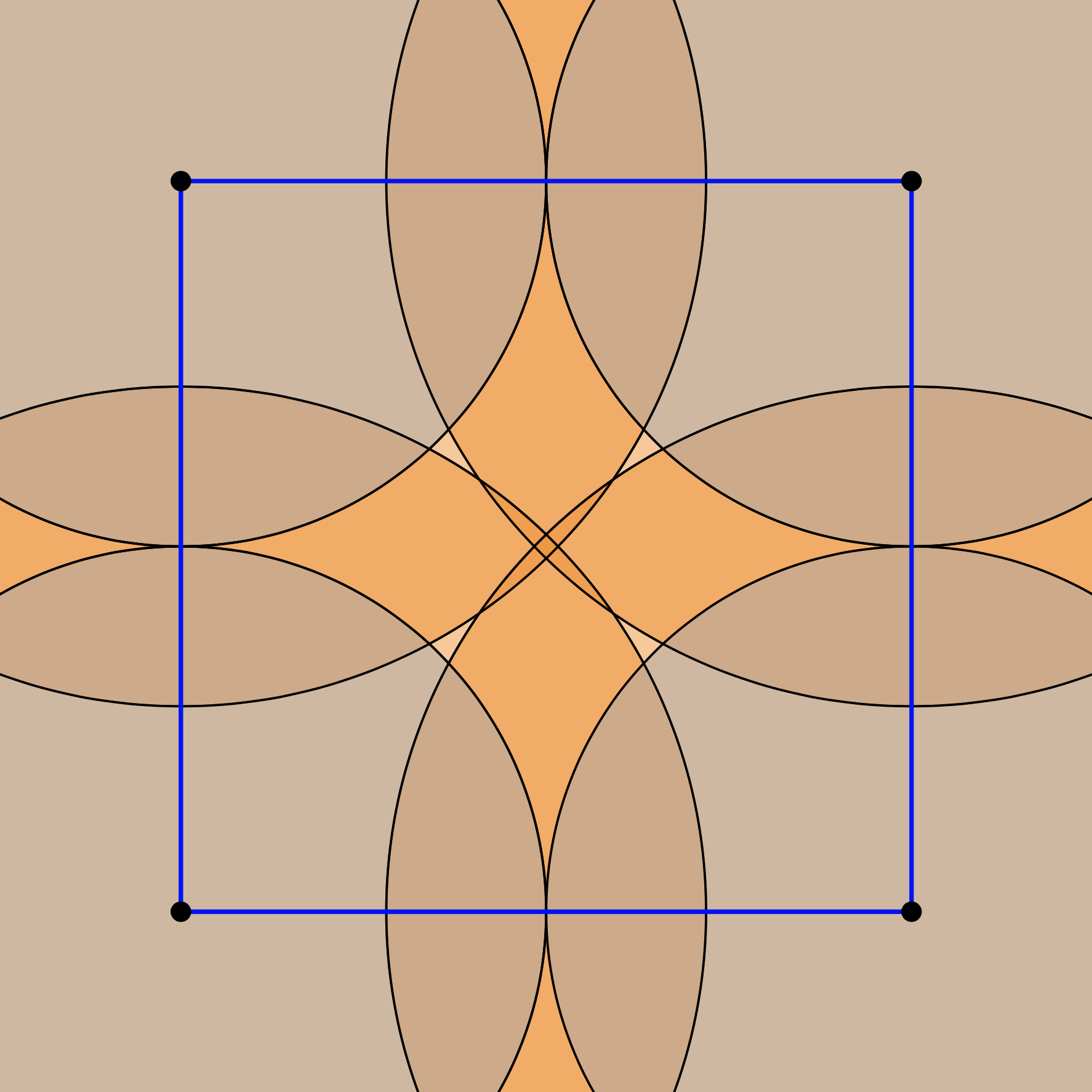}\hspace{1em}
    \caption{ }
    \label{subfig-3:dummy}
\end{subfigure}
	\end{center}
\caption{Examples of hard discs with depletion shell; the shaded areas represent the overlap between the depletion shells.
{(a) For $\rho$ smaller but close to the critical case $\rho_2$, where only pair interactions occur;}
{(b) For $\rho_2<\rho\leq\rho_3$, there cannot be four-body interactions; (c) For $\rho >\rho_3$, four-body interactions occur.}}\label{fig:3spheres}
\end{figure}

\begin{proof}
We shortly recall why the $k$-body interactions $\phidep_k, k \geq 3$, vanish if and only if $\rho \leq \rho_2$.
Indeed, in any dimension $d\geq 2$, the critical value of $\rho$ for multiple depletion shells to overlap can be computed by considering $3$ spheres whose centres lie on the vertices of an equilateral triangle of side-length $2\, \rgrande$, see Figure \ref{fig:3spheres} (a).
There is overlap as soon as the centre of the triangle is at distance $\rdep= \rgrande (1+\rho)$ from the vertices, that is if and only if
 $  \rho >\rho_2\defeq \frac{2}{3}\sqrt{3}-1\simeq 0.1547$.

The value for $\rho_3$ depends on the dimension $d$.
In dimension $d=2$, the lowest value of $\rho$ for more than three depletion shells to overlap is obtained by considering four spheres whose centres lie on the vertices of a square of side length $2\rgrande$, see Figure \ref{fig:3spheres} (b-c).
So, there is a $4$-way overlap if and only if
\begin{equation*}
    2\sqrt{2}\, \rgrande < 2\rdep \iff \rho > \rho_3\defeq \sqrt{2}-1\simeq 0.4142 \, .
\end{equation*}
In dimension $d\geq 3$, the smallest value of $\rho$ is obtained by considering four spheres whose centres lie on the vertices of a regular tetrahedron of side-length $2\rgrande$. So, there is a $4$-way overlap if and only if the distance between any vertex and the center of mass of the tetrahedron is smaller than $\rdep$, that is
\begin{equation*}
    \sqrt{3/2}\, \rgrande < \rdep \iff \rho > \rho_3\defeq \sqrt{3/2}-1 \simeq 0.2247 \, .
\end{equation*}
This concludes the proof.
\end{proof}

Let us compute the pair depletion interaction $\phidep_2$ in any dimension.

\begin{lemma} \label{def:depletion_potential_dimd}
The pair depletion interaction is attractive, radial and acts only at very close range. More precisely, its expression in dimension $d$ is given by the following integral: for any $\ x_i,x_j \in \R ^d$,
\begin{align*}
\phidep_2(x_{i},x_{j}) = -  
2\,v_{d-1} ~ \rdep\rspace ^d \int_0^{\arccos(u)}  (\sin\theta)^d \,d\theta 
\ \un_{[\frac{1}{1+\rho},1]} (u) , 
\text{ where } u\defeq \frac{|x_i - x_j|}{ 2 \rdep}.
\end{align*}
In particular, $\phidep_2(x_{i},x_{j})\neq 0$ only if $|x_i - x_j|$ is between $2\rgrande$ and $2\rdep$.
\end{lemma}
\begin{proof}
Note first that the function $\phidep_2(x_{i},x_{j}) = - \vol{B(x_i,\rdep)\cap B(x_j,\rdep)}$ is non-positive and therefore the pair depletion interaction is attractive.\\
Moreover, it only depends on the distance between both points $x_i$ and $x_j$ in $\R ^d$. 
Therefore, to simplify the notations we define the (rescaled) function $\Vovlap$  as follows:
\begin{equation}\label{def:depletion_potential}
\Vovlap(u)\defeq - \phidep_2(x_{i},x_{j}) 
  \textrm{ with } u= \frac{|x_i - x_j|}{ 2 \rdep} \in \Big[\frac{\rgrande}{\rdep},+\infty \Big) .
\end{equation}
The function $\Vovlap$ is clearly decreasing from its maximal value $\Vovlap^*$, attained at $u= \frac{\rgrande}{\rdep}=\frac{1}{1+\rho}$, to its minimal value $0$, attained at $u=1$. It vanishes on $[1,+\infty)$, underlining the strong locality of this pair interaction. 
We now compute it.

By shift invariance and symmetry, denoting by $(e_1,\ldots,e_d)$ an orthonormal basis of $\R^d$, we find
\begin{align*}
\Vovlap(u) 
& = \vol{B(0,\rdep)\cap B(2\rdep u ~ e_1,\rdep)} \\
& = 2\, \vol{ \left\{ x\in \R^d, |x|\leq \rdep \text{ and } x\cdot e_1\ge \rdep u \right\} }\\
& = 2\, \vol{ \bigg\{ x\in\R^d, \sum_{i=2}^d (x\cdot e_i)^2 \le \rdep\rspace ^2 - (x\cdot e_1)^2 
\text{ and }  x\cdot e_1\ge \rdep u  \bigg\} } .
\end{align*}
Introducing $v_{d-1}$, the volume of the unit sphere in $\R^{d-1}$, one gets, for $\frac{1}{1+\rho} \leq u\leq  1$, 
\begin{align*}
\Vovlap(u) 
& = 2\int_{\rdep u}^{\rdep} v_{d-1} \left( \sqrt{\rdep\rspace ^2 -t^2} \right)^{\hspace{-2mm}{d-1}} ~ dt 
= 2v_{d-1} ~ (\rdep\rspace^{2})^{\frac{d-1}{2}} \int_{u}^{1} \left( \sqrt{ 1-s^2} \right)^{\hspace{-1mm}{d-1}} \rdep \, ds \\
&= 2v_{d-1} ~ \rdep\rspace ^{d} \int_0^{\arccos(u)} \left( \sin\theta \right)^d \, d\theta. \qedhere
\end{align*}
\end{proof}

\begin{example}[\textbf{The pair depletion interaction between discs in $\R^2$}] \label{ex:depinterdR2}
{\hspace{.1em}\\[1mm]}

Applying Lemma \ref{def:depletion_potential_dimd}, one can compute explicitly the integral in the case $d=2$ to obtain, for $u= \frac{|x_i - x_j|}{ 2 \rdep}\in [\frac{1}{1+\rho},1]$,
\begin{equation*}
\phidep_2(x_{i},x_{j}) = - 2\,v_{1} ~ \rdep\rspace ^2 
\int_0^{\arccos(u)}  (\sin\theta)^2 \,d\theta  = - 2\rdep\rspace ^2 \, \big(\arccos(u) - u\sqrt{1-u^2}\, \big).
\end{equation*}
It can also be computed directly via the following simple planar geometry argument (cf. Figure~\ref{fig:area}).

\begin{figure}
\begin{center}
    \includegraphics[width=.3\textwidth]{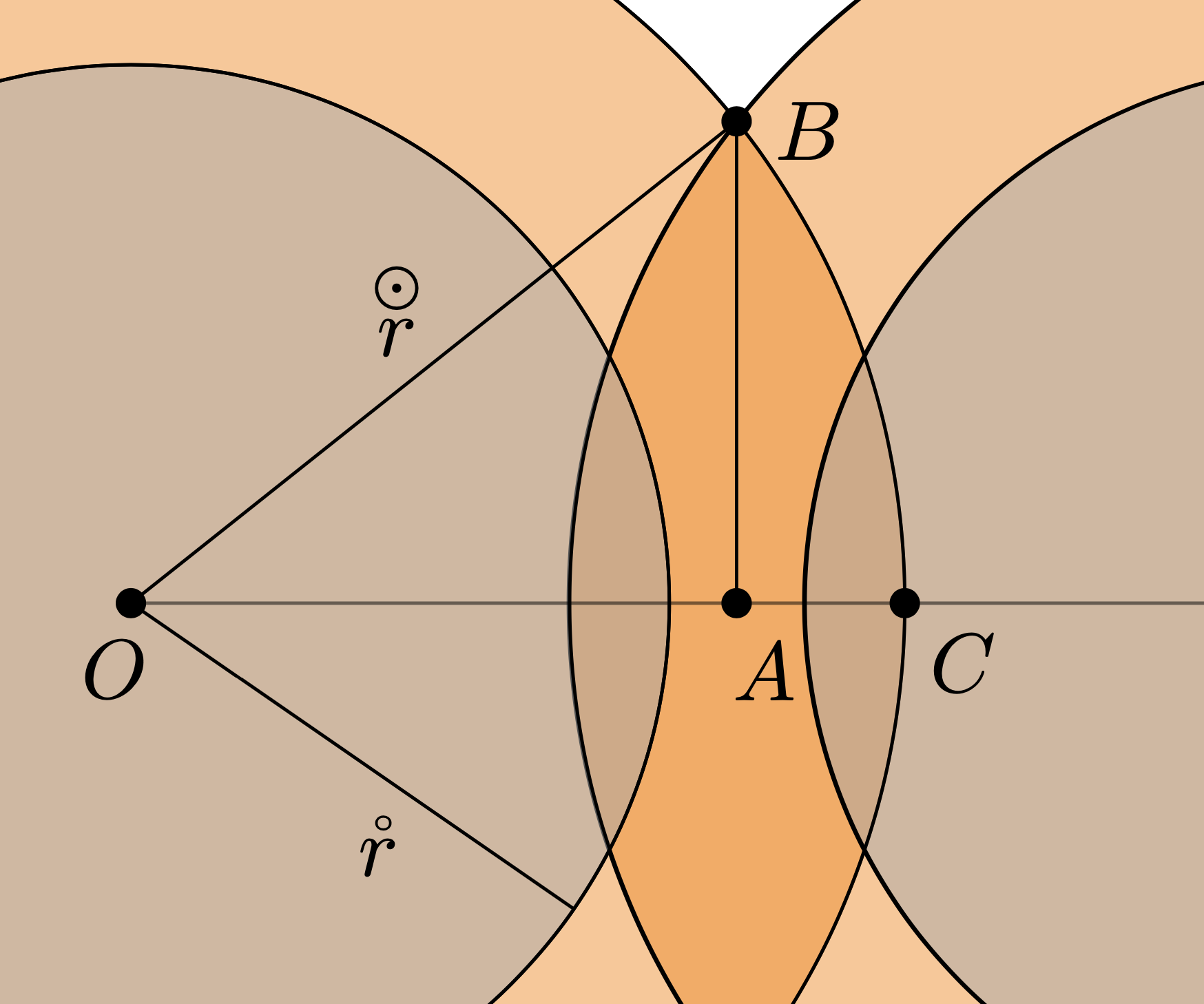}
\end{center}
\caption{Intersection of two depletion discs.}\label{fig:area}
\end{figure}  

In the plane, consider a disc centred at a point $O \in \R^2$, with radius $\rdep$, and suppose it intersects another disc of the same radius and whose centre is at distance $2\rdep\, u$ from $O$.

Denote by $\mathcal{A}_{\text{sec}}(u)$, the area of the circular sector in $OBC$. Consider then
the area $\mathcal{A}_{\text{tri}}(u)$ of the rectangular triangle $OBC$. One has
\begin{equation*}
    \mathcal{A}_{\text{sec}}(u) = \frac{1}{2}\,\arccos(u) \,\rdep\rspace ^2,
    \quad \mathcal{A}_{\text{tri}}(u) = \frac{1}{2}\overline{OA}\ \overline{AB} 
    =\frac{\rdep\,\! ^2}{2}\, u\sqrt{1-u^2}.
\end{equation*}
The area of the overlap is then given by 
\begin{equation} \label{eq:Vovlapd2}
 \Vovlap(u) = 4\big(\mathcal{A}_{\text{sec}}(u) - \mathcal{A}_{\text{tri}}(u)\big) 
 = 2\rdep\rspace ^2 \big( \arccos(u) - u\sqrt{1-u^2}\, \big) ,
\end{equation}
see Figure \ref{fig:graphd2}, left.
The maximal overlap area is
\begin{equation} \label{eq:maxVovlapd2}
\begin{aligned}[t]
\Vovlap^*
= \Vovlap\Big(\frac{1}{1+\rho}\Big)
 &=
  2\rdep\rspace ^2 \Big( \arccos\big(\frac{1}{1+\rho}\big) - \frac{1}{(1+\rho)^2}\sqrt{\rho (2 + \rho)}\, \Big).
\end{aligned}
\end{equation}
Moreover, the derivative on $(\frac{1}{1 + \rho}, 1)$ of the function $\Vovlap$ is given by
\begin{equation*}
   \Vovlap'(u) = 
    - 4 \rdep\rspace ^2  \sqrt{1-u^2}
\end{equation*}
which vanishes in  $1$ (see Figure \ref{fig:graphd2}, right). Therefore, the function $\Vovlap$ is of class $C^1$ over the full interval $[\frac{1}{1 + \rho},+\infty)$.
Nevertheless, its second derivative explodes at the point $u=1$, since
\begin{equation*}
   \Vovlap''(u) = 
     4 \, \rdep\rspace ^2  \frac{u}{\sqrt{1-u^2}},\quad \frac{1}{1 + \rho} \leq u < 1.
\end{equation*}

\begin{figure}
\vspace{.5em}
   \begin{center}
\begin{subfigure}[b]{.45\textwidth}
\begin{tikzpicture}
\begin{groupplot}[
    group style={
        group name=my fancy plots,
        group size=2 by 1,
        xticklabels at=edge bottom,
        horizontal sep=-4pt
    },
    width=.45\textwidth,
    xmin=0, xmax=6
]
\nextgroupplot[xmin=-3,xmax=15,
			   ymin=-.03,ymax=1,
               ytick={0},
               axis x line* = middle,
               xtick=\empty,
               xticklabels=\empty,
               extra x tick style = {xshift = 1.92ex,yshift=-0.5ex},
               extra y ticks = {0.59, 0.98},
               extra y tick style={tickwidth=0mm},
               extra y tick labels = {\footnotesize $\Vovlap^*$, \footnotesize $\Vovlap(u)$},
               axis y line=middle,
               height=5cm,
               width=2.2cm,
               enlargelimits=0.02
               ]
\draw[lightgray] (6,6.2) -- (170,6.2); 
\addplot[domain=0:9.2,thick,blue] {28};
\nextgroupplot[xmin=0.65,xmax=1.1,
	           ymin=-.03,ymax=1,
               ytick={0,10,20},
               xlabel = \footnotesize $u$,
               xtick=\empty,
               xticklabels=\empty,
               extra x ticks={0.7,1},	
               extra x tick labels = {\footnotesize $\frac{1}{1+\rho}$,\footnotesize 1},
               hide y axis,
               axis y line=middle,
               axis x line=middle,
               height=5cm,
               width=6cm,
               enlargelimits=0.02
               ]
\node[inner ysep=3pt,inner xsep=1.5pt,fill=white](D)at(1,3){};
\draw(D.south west)--(D.north west) (D.south east)--(D.north east);
\addplot[mark=*,mark size=1pt,blue] coordinates{(0.7,0.59)} node[anchor=south west]{};
\addplot[domain=0.7:1,samples=200,thick,blue] {2*rad(acos(\x)) - 2*\x*sqrt(1-\x^2)};
\addplot[domain=1:1.09,samples=2,thick,blue] {0};
\draw[lightgray] (50,0) -- (50,62); 
\draw[lightgray] (0,62) -- (50,62);               
\end{groupplot}
\end{tikzpicture}
\end{subfigure}
\begin{subfigure}[b]{.45\textwidth}
\begin{tikzpicture}
\useasboundingbox (-1.5,3) rectangle (5,-.67);
\begin{groupplot}[
    group style={
        group name=my fancy plots,
        group size=2 by 1,
        xticklabels at=edge bottom,
        horizontal sep=-4pt
    },
    width=.45\textwidth,
    xmin=0, xmax=6
]
\nextgroupplot[xmin=-3,xmax=15,
			   ymin=-3,ymax=.8,
               ytick={0},
               axis x line* = middle,
               xtick=\empty,
               xticklabels=\empty,
               extra x tick style = {xshift = 1.92ex,yshift=-0.5ex},
               extra y ticks = {.68},
               extra y tick labels = {\footnotesize $\Vovlap'(u)$},
               extra y tick style={tickwidth=0mm},
               axis y line=middle,
               height=5cm,
               width=2.2cm,
               enlargelimits=0.02
               ]
\addplot[domain=0:9.2,thick,blue] {28};         

\nextgroupplot[xmin=0.65,xmax=1.1,
	           ymin=-3,ymax=.8,
               ytick={0,10,20},
               xlabel = \footnotesize $u$,
               xtick=\empty,
               xticklabels=\empty,
               extra x ticks={0.7,1},
               extra x tick style={tick label style = {above, yshift=.5ex}},	
               extra x tick labels = {\footnotesize $\frac{1}{1+\rho}$,\footnotesize 1},
               hide y axis,
               axis y line=middle,
               axis x line=middle,
               height=5cm,
               width=6cm,
               enlargelimits=0.02
               ]
\node[inner ysep=3pt,inner xsep=1.5pt,fill=white](D)at(1,300){};
\draw(D.south west)--(D.north west) (D.south east)--(D.north east);
\addplot[mark=*,mark size=1pt,blue] coordinates{(0.7,-2.85)} node[anchor=south west]{};
\addplot[domain=0.7:1,samples=200,thick,blue] {-4*sqrt(1-\x^2)};
\addplot[domain=1:1.09,samples=2,thick,blue] {0};
\draw[lightgray] (50,13) -- (50,300); 
\end{groupplot}
\end{tikzpicture}
\end{subfigure}
\end{center}
\caption{Behaviour of $u \mapsto \Vovlap(u)$ (left) and  $u \mapsto \Vovlap'(u)$ (right) for discs in the plane.}
    \label{fig:graphd2}
\end{figure}
    
\begin{remark}
If the size of the particles is large enough ($\rho > \rho_2$), then a three-body depletion interaction can appear. Its explicit computation is done in~\cite[Equation~2.2]{K78}:
\begin{eqnarray*}
    \phidep_3(x_1,x_2,x_3) &=&  
\frac{1}{2} \bigg( \sum_{ \{i,j\} \subset \{1,2,3\}}  \Vovlap\Big(\frac{|x_i - x_j|}{2 \rdep}\Big)   - \pi \rdep\rspace ^2 \\
&& + 
\frac{1}{2} \sqrt{4 |x_1 - x_2|^2|x_1 - x_3|^2 - 
\big(|x_1 - x_2|^2+|x_1 - x_3|^2 - |x_2 - x_3|^2\big)^2} \,  \bigg).    
\end{eqnarray*}    
\end{remark}
\end{example}
\begin{example}[\textbf{The pair depletion interaction between balls in $\R^3$}] \label{ex:depinterdR3}
{\hspace{.1em}\\[1mm]}

Applying Lemma \ref{def:depletion_potential_dimd} in the case $d=3$, one can explicitly compute the integral and obtain, for $u\in [\frac{1}{1 + \rho},1]$:
\begin{equation*} \label{eq:Vovlapd3}
 \Vovlap (u)
=  2\,v_{2} ~ \rdep\rspace ^3 \int_0^{\arccos(u)}  (\sin\theta)^3 \,d\theta  
=  2 \pi\,\rdep\rspace ^3 \int_u^1  (1-s^2) \,ds
=  \frac{4\pi}{3} \rdep\rspace ^3 ( 1 -  u)^2 (1+\frac{u}{2}),
\end{equation*}
see also, e.g.,~\cite{Vr76}. 
Its maximal value is given by
\begin{equation} \label{eq:maxVovlapd3}
    \Vovlap^*=  \Vovlap\Big(\frac{1}{1 + \rho}\Big)= 
2\pi \, \rgrande\rspace ^3 \, \rho^2 \big(1 + \frac{2}{3} \rho\big).
\end{equation}
Its first derivative satisfies
\begin{equation*}
   \Vovlap '(u) = - 2\, \pi \, \rdep\rspace ^3 \big( 1 -  u^2 \big), \quad \frac{1}{1 + \rho}< u < 1.
\end{equation*}
This expression vanishes in  $1$ which implies -- as in dimension $2$ -- the $C^1$-regularity of the function $\Vovlap$ over the full intervall $ [\frac{1}{1 + \rho},+\infty)$, see Figure \ref{fig:graphd3}, left.
Its second derivative  on $(\frac{1}{1 + \rho}, 1)$ satisfies $ \Vovlap''(u) = 4\, \pi \, \rdep\rspace ^3 u$ which does not vanish at $u=1$ but remains bounded, contrary to the behaviour in dimension $2$, see Figure \ref{fig:graphd3}, right.

\begin{figure}[ht]
\vspace{1em}
\begin{center}
\begin{subfigure}[b]{.45\textwidth}
\begin{tikzpicture}
\begin{groupplot}[
    group style={
        group name=my fancy plots,
        group size=2 by 1,
        xticklabels at=edge bottom,
        horizontal sep=-4pt
    },
    width=.45\textwidth,
    xmin=0, xmax=6
]
\nextgroupplot[xmin=-3,xmax=15,
			   ymin=-.03,ymax=.2,
               ytick={0},
               axis x line* = middle,
               xtick=\empty,
               xticklabels=\empty,
               extra x tick style = {xshift = 1.92ex,yshift=-0.5ex},
               extra y ticks = {.121,.195},
               extra y tick labels = {\footnotesize $\Vovlap^*$, \footnotesize $\Vovlap(u)$},
               extra y tick style={tickwidth=0mm},
               axis y line=middle,
               height=5cm,
               width=2.2cm,
               enlargelimits=0.02
               ]
\draw[lightgray] (0,1.51) -- (170,1.51); 
\addplot[domain=0:9.2,thick,blue] {28};
\nextgroupplot[xmin=0.65,xmax=1.1,
	           ymin=-.03,ymax=.2,
               ytick={0},
               xlabel = \footnotesize $u$,
               xtick=\empty,
               xticklabels=\empty,
               extra x ticks={0.7,1},	
               extra x tick labels = {\footnotesize $\frac{1}{1+\rho}$,\footnotesize 1},
               hide y axis,
               axis y line=middle,
               axis x line=middle,
               height=5cm,
               width=6cm,
               enlargelimits=0.02
               ]
\node[inner ysep=3pt,inner xsep=1.5pt,fill=white](D)at(1,30){};
\draw(D.south west)--(D.north west) (D.south east)--(D.north east);
\addplot[mark=*,mark size=1pt,blue] coordinates{(0.7,.121)} node[anchor=south west]{};
\addplot[domain=0.7:1,samples=200,thick,blue] {(1-\x)^2 * (1+0.5*\x)};
\addplot[domain=1:1.09,samples=2,thick,blue] {0};
\draw[lightgray] (50,25) -- (50,151); 
\draw[lightgray] (0,151) -- (50,151);               
\end{groupplot}
\end{tikzpicture}
\end{subfigure}
\begin{subfigure}[b]{.45\textwidth}
\begin{tikzpicture}
\useasboundingbox (-1.5,3) rectangle (5,-.4);
\begin{groupplot}[
    group style={
        group name=my fancy plots,
        group size=2 by 1,
        xticklabels at=edge bottom,
        horizontal sep=-4pt
    },
    width=.45\textwidth,
    xmin=0, xmax=6
]
\nextgroupplot[xmin=-3,xmax=15,
			   ymin=-3.5,ymax=.8,
               ytick={0},
               axis x line* = middle,
               xtick=\empty,
               xticklabels=\empty,
               extra y ticks = {.61},
               extra y tick labels = {\footnotesize $\Vovlap'(u)$},
               extra y tick style={tickwidth=0mm},
               extra x tick style = {xshift = 1.92ex,yshift=-0.5ex},
               axis y line=middle,
               height=5cm,
               width=2.2cm,
               enlargelimits=0.02
               ]
\addplot[domain=0:9.2,thick,blue] {28};         
%
\nextgroupplot[xmin=0.65,xmax=1.1,
	        ymin=-3.5,ymax=.8,
                xlabel = \footnotesize $u$,
                ytick={0,10,20},
                xtick=\empty,
                xticklabels=\empty,
                extra x ticks={0.7,1},
                extra x tick style={tick label style = {above, yshift=.5ex}},	
                extra x tick labels = {\footnotesize $\frac{1}{1+\rho}$,\footnotesize 1},
                hide y axis,
                axis y line=middle,
                axis x line=middle,
                height=5cm,
                width=6cm,
                enlargelimits=0.02]
\node[inner ysep=3pt,inner xsep=1.5pt,fill=white](D)at(1,350){};
\draw(D.south west)--(D.north west) (D.south east)--(D.north east);
\addplot[mark=*,mark size=1pt,blue] coordinates{(0.7,-3.2)} node[anchor=south west]{};
\addplot[domain=0.7:1,samples=200,thick,blue] {-2*pi*(1-\x^2)};
\addplot[domain=1:1.09,samples=2,thick,blue] {0};
\draw[lightgray] (50,30) -- (50,350); 
\end{groupplot}
\end{tikzpicture}
\end{subfigure}
\end{center}
\caption{Behaviour of $u \mapsto \Vovlap(u)$ (left) and  $u \mapsto \Vovlap'(u)$ (right) for balls in $\R^3$.}
\label{fig:graphd3}
\end{figure}

Notice that, if the size of the particles is large enough ($\rho > \rho_2$), then a three-body depletion interaction can appear. K.~W.~Kratky computes it implicitly in~\cite[Equation~1.2b]{K81}.
\end{example}
\subsection{An associated gradient dynamics} \label{sec:dep_dynamics}

From now on, we suppose $\rho \leq\rho_2$, so that the energy of a configuration of hard spheres defined by \eqref{eq:inex} is only generated by pairwise interactions.  
Due to Lemma \ref{def:depletion_potential_dimd} and the definition of the overlap function \eqref{def:depletion_potential}, the energy function becomes
\begin{equation} \label{eq:nrgparpaires}
    \nrg (\bgrande) = n v_{d} \,  \rdep\rspace ^d  - \sum_{1\leq i<j\leq n}  \Vovlap\Big(\frac{|\grande_i - \grande_j|}{2 \rdep}\Big).
\end{equation}

Therefore, using the explicit expression of $\Vovlap$ and $\Vovlap'$, the gradient field of the energy 
satisfies, for $i \in \{1,\cdots,d\}$,
\begin{align} \label{eq:gradnrg}
\nabla_i \nrg(\bgrande) &= - \frac{1}{2\rdep}
\sum_{j=1}^n \Vovlap '\Big(\frac{|\grande_i-\grande_j|}{2\rdep}\Big) \frac{\grande_i -\grande_j}{|\grande_i -\grande_j|}\nonumber\\
&= v_{d-1} ~ \rdep\rspace ^{d-1} 
    \sum_{j=1}^n \Big( 1-\frac{|\grande_i -\grande_j|^2}{4\rdep\rspace^2} \Big)_{\hspace{-1mm}+}^{\hspace{-1mm}\frac{d-1}{2}} 
                        \   \frac{\grande_i -\grande_j}{|\grande_i -\grande_j|} ,
\end{align}
where, as usual, $u_+\defeq \max (0,u)$ denotes the positive part of any real number $u$. 

Let us now define the diffusive dynamics of the hard spheres $(\Grande_i )_{i\leq n}$ submitted to ($\zpiccola$ times) this gradient field. It solves the following SDE:
\begin{equation}\tag{${\mathfrak S}_{n}^\textrm{dep}$}\label{eq:Sndep}
\begin{cases}
	\begin{array}{l}
  \textrm{for } i, j \in \{1,\dots,n\}, \, t \in [0,1] ,                             \\ 
  d \Grande_i (t) = \sigmaG\, d \WGrande_i(t) - \disp \frac{1}{2} 
  \, \nabla\psiG \big(\Grande_i (t)\big)\, d t \\ 
  \phantom{d \Grande_i (t) = 
  d W_i(t) -  }
  - \disp \frac{\zpiccola}{2} \,
  v_{d-1} ~ \rdep\rspace ^{d-1} 
    \sum_{j=1}^n \Big( 1-\frac{|\Grande_i -\Grande_j|^2}{4\rdep\rspace ^2} \Big)_{\hspace{-1mm}+}^{\hspace{-1mm}\frac{d-1}{2}} 
                          \frac{\Grande_i -\Grande_j}{|\Grande_i -\Grande_j|}\, d t  
                          \\
  \phantom{d \Grande_i (t) = 
  d W_i(t) - \frac{1}{2} 
  \nabla\psiG \big(\Grande_i (t)\big)\, d t  + \disp \frac{\zpiccola}{2} \,
  v_{d-1}}
+ \disp \sum_{j=1}^n\big(\Grande_i (t)-\Grande_j(t)\big) d L_{ij}(t) \, 
, \\
\disp
  L_{ij}(0) = 0 , \quad L_{ij} \equiv L_{ji}, \quad
  L_{ij}(t) = \int_0^t \un_{|\Grande_i (s)-\Grande_j(s)|=2 \,\rgrande} \, d  L_{ij}(s), \quad L_{ii} \equiv 0 ,
  \end{array}
\end{cases}	
\end{equation}
where $\WGrande_1,...,\WGrande_n$ are $n$ independent $\R^d$-valued Brownian motions and $L_{ij}$ denotes the collision local time between the sphere $i$ and the sphere $j$. The local times $L_{ij}$ describe the effects of the elastic collision between the hard spheres $i$ and $j$ (subject to normal reflection).
\begin{theorem}
The SDE \eqref{eq:Sndep} admits a unique solution in the set of admissible configurations $\D$. 
Moreover the measure $\mugrande_{\zpiccola}$ defined by \eqref{eq:muronde} is reversible under this dynamics.
\end{theorem}

\begin{proof}

The stochastic differential system \eqref{eq:Sndep} describes the dynamics of an $nd$-di\-men\-sio\-nal gradient diffusion with reflection at the boundary of the set of admissible hard spheres
  $  \D \cap \MGrande = \bigcap_{1\le i<j \le n}
    \Big\{\bgrande\in \MGrande : \, \Gamma_{ij}(\bgrande) \geq 0 \Big\} $.
This  domain is induced by the pairwise constraint functions $ \Gamma_{ij}(\bgrande)\defeq\frac{|\grande_i - \grande_j|^2 }{ 4 \rgrande\rspace ^2} -1$ introduced in \eqref{def:Gammaij}.

We first verify the smoothness of each constraint function and their global compatibility: this was already done in steps (i) and (ii) of the proof of Proposition \ref{prop:existencesolrevSnmR} for a larger number of constraints.

Existence and uniqueness 
of a strong solution to \eqref{eq:Sndep} are then ensured by  \cite[Theorem 2.2]{MultipleConstraint} as soon as the gradient field of the energy function is Lipschitz continuous and bounded on the domain $ \D \cap \MGrande$. 
Considering the expression \eqref{eq:gradnrg}, this is the case in any dimension $d>2$.

In dimension $d=2$, however, the gradient
\begin{equation*} \label{eq:gradnrgd2}
\nabla_i \nrg(\bgrande) = 2\, \rdep \sum_{j=1}^n \sqrt{ \bigg(1-\frac{|\grande_i -\grande_j|^2}{4\rdep\rspace^2} \bigg)_{\hspace{-1mm}+}}\   \frac{\grande_i -\grande_j}{|\grande_i -\grande_j|}
\end{equation*}
is Lipschitz continuous on $ \D\cap \MGrande$, except around the configurations containing a pair of hard spheres whose centres lie at the critical distance $2\,\rdep=2\, \rgrande (1+\rho)$. This pathology corresponds to the explosion of the function $\Vovlap''$ at $u=1$ observed in Example \ref{ex:depinterdR2}. Nevertheless, the SDE \eqref{eq:Sndep} can be solved straightforwardly as follows. 

The solution is constructed progressively on intervals defined through a sequence of stopping times indicating, either (i) the moment in which a pair $i,j$ of hard spheres is close enough, say $\Gamma_{ij}(\bgrande) \in [0,\rho]$, or (ii) the moment in which two spheres are close to the critical distance, say $\Gamma_{ij}(\bgrande) \in [2 \rho,2 \rho (1 + \rho/2)]$. In the first kind of time interval, the energy function is smooth and we can use the existence result for $d>2$. In the second kind of time interval, the energy function is not smooth anymore, but the hard spheres cannot collide, meaning that \eqref{eq:Sndep} does not contain collision local times. Therefore, we can apply strong existence results for Brownian diffusions with bounded Borel drift as, e.g.,~\cite[Theorem 1]{Ve81}. 

Applying \cite[Theorem 2.5]{MultipleConstraint} (see also \cite{ST87}), we get that $\e^{-\zpiccola \nrg(\bgrande)} \,\un_{\D}(\bgrande)\, d\bgrande$ is a time-reversible measure for the dynamics \eqref{eq:SnmR}. This concludes the proof of the theorem.
\end{proof}

\smallskip

\subsection{High-density regime of the small-particle bath: towards an optimal packing of finitely-many hard spheres}\label{sec:packing}

In Proposition \ref{prop:tracemesureequ}, we saw that the $\log$-density of the reversible measure $\mugrande_{\zpiccola}$ of the $n$ hard  spheres is proportional to the activity $\zpiccola$ of the bath of particles (i.e., the medium) in which they evolve. It is therefore natural to consider the asymptotic behaviour of the measure $\mugrande_{\zpiccola}$ in a high-density regime of the bath, that is for $\zpiccola$ tending to $\infty$.

Heuristically, for a fixed activity $\zpiccola$, the measure $\mugrande_{\zpiccola}$ favours the configurations with low energy $\nrg$. As $\zpiccola$ increases, $\mugrande_{\zpiccola}$ will concentrate more and more on hard  spheres configurations with minimal energy, as expressed in Proposition \ref{prop:packing} below.

\smallbreak

Let us first introduce the \emph{contact number} $c_n(\bgrande)$ of an admissible configuration $\bgrande$ of $n$ spheres with radius $\rgrande$ as the number of its pairwise contacts:
\begin{equation*}
    c_n(\bgrande) \defeq \#\{ (i,j) \text{ such that } |\grande_{i}-\grande_{j}|=2\rgrande,\, 1\leq i<j\leq n\}.
\end{equation*}
\begin{proposition}\label{prop:packing}
Assume that $\rho \leq \rho_2$ so that the energy function $\nrg$  contains only pair interaction terms. Asymptotically in $\zpiccola$, the reversible probability measure $\mugrande_{\zpiccola}$ defined in \eqref{eq:muronde}
concentrates around admissible configurations which minimise the energy.
More precisely, let 
$\nrg^*_n\defeq \displaystyle \inf \{ \nrg (\bgrandey): \bgrandey\textrm{ is an admissible $n$-sphere configuration} \}$. 
Then
\begin{equation}\label{eq:etaconcentration}
    \forall \varepsilon, \eta >0 \quad \exists \zpiccola_c>0 \quad \forall \zpiccola>\zpiccola_c \quad  
    \mugrande_{\zpiccola}\Big(\Big\{\bgrandey \in \D: \ \nrg (\bgrandey) \leq \nrg_n^* + \eta \Big\}\Big) \geq 1 - \varepsilon .
\end{equation}
Moreover, admissible $n$-sphere configurations who realise the minimal energy $\nrg_n^*$ maximise the contact number $c_n$.
\end{proposition}

\begin{proof}
We first prove \eqref{eq:etaconcentration}.

Since $Z_\zpiccola\defeq \displaystyle \int_{\D} \e^{- \zpiccola \nrg(\bgrande )} \lambda^{\otimes n} (d \bgrande)$  is a normalisation constant, one has
\begin{align*}
\mugrande_{\zpiccola}\big(\{\bgrande \in \D: \ \nrg (\bgrande) \leq \nrg_n^* + \eta \}^c\big)  
& = \frac{\int_{\D}  \e^{- \zpiccola \nrg(\bgrande )}  \  \un_{\nrg (\bgrande) > \nrg_n^* + \eta} \ \lambda^{\otimes n} (d \bgrande)}{\int_{\D} \e^{- \zpiccola \nrg(\bgrande )}  \lambda^{\otimes n} (d \bgrande)}\\
& = \frac{\int_{\D} \e^{- \zpiccola \big(\nrg(\bgrande ) - (\nrg_n^* + \eta) \big)}  \  \un_{\nrg (\bgrande) > \nrg_n^* + \eta} \
  \lambda^{\otimes n} (d\bgrande)}{\int_{\D} \e^{- \zpiccola \big(\nrg(\bgrande ) - (\nrg_n^* + \eta) \big)}  \lambda^{\otimes n} (d\bgrande)}\\
  & \leq  \frac{\int_{\D} \e^{- \zpiccola \big(\nrg(\bgrande ) - (\nrg_n^* + \eta) \big)}  \  \un_{\nrg (\bgrande) > \nrg_n^* + \eta} \ \lambda^{\otimes n} (d \bgrande)}
  {\int_{\D} \e^{- \zpiccola \big(\nrg(\bgrande ) - (\nrg_* + \eta) \big)} \un_{\nrg (\bgrande) \leq \nrg_n^* + \eta} \lambda^{\otimes n} (d \bgrande)}\\
  & \leq  \frac{\int_{\D} \e^{- \zpiccola \big(\nrg(\bgrande ) - (\nrg_n^* + \eta) \big)}  \  \un_{\nrg (\bgrande) > \nrg_n^* + \eta} \ \lambda^{\otimes n} (d \bgrande)}
  {\lambda^{\otimes n} \big(\{\bgrande \in \D: \nrg (\bgrande) \leq \nrg_n^* + \eta\}\big)} 
  \eqdef\frac{\boldsymbol{n}(\eta,\zpiccola)}{\boldsymbol{d}(\eta)}.
\end{align*}
Note that the denominator $\boldsymbol{d}(\eta)$ is positive and does not depend on $\zpiccola$. For the numerator, note that for each fixed $\eta$, pointwise in $\bgrande$,
\begin{equation*}
    \lim_{\zpiccola\nearrow \infty} \e^{- \zpiccola \big(\nrg(\bgrande ) - (\nrg_n^* + \eta) \big)}  \  \un_{\nrg (\bgrande) > \nrg_n^* + \eta}=0
\end{equation*}
Since $\bgrande \mapsto \e^{- \zpiccola \big(\nrg(\bgrande ) - (\nrg_n^* + \eta) \big)}  \  \un_{\nrg (\bgrande) > \nrg_n^* + \eta}$ is uniformly bounded by the constant $1$, which is $\lambda^{\otimes n}$-integrable, the dominated convergence theorem implies that the numerator $\boldsymbol{n}(\eta,\zpiccola)$ vanishes as $\zpiccola$ increases. This completes the proof of \eqref{eq:etaconcentration}.
\smallbreak

Next, we investigate the shape of the admissible $n$-sphere configurations whose energy realises the minimum $\nrg_n^*$. Recalling \eqref{eq:nrgparpaires},
\begin{align*} 
     \nrg (\bgrande) = n v_{d} \,  \rdep\rspace ^d  - \sum_{1\leq i<j\leq n}  \Vovlap\Big(\frac{|\grande_i - \grande_j|}{2 \rdep}\Big).
\end{align*}
Thus,
\begin{equation} \label{minnrg}
\begin{aligned}[t]
     \nrg_* = \inf_{\bgrande \in \D}\nrg (\bgrande) &= n v_{d} \,  \rdep\rspace ^d - \max_{\bgrande \in \D}\sum_{1\leq i<j\leq n} \mathcal{V}_{\text{ovlap}}\Big(\frac{|\bgrande_i - \bgrande_j|}{2 \rdep}\Big)\\
      &= n v_{d} \,  \rdep\rspace ^d - \mathsf{c}(n,d) \  \Vovlap^*
      \,,
\end{aligned}
\end{equation} 
where 
\begin{equation*}
     \mathsf{c}(n,d) \defeq \max \{ c_n(\bgrande), \ \bgrande=\{\grande_1, \dots,\grande_n\} \in \D \subset\R^d\}
\end{equation*}
is the maximum of the contact number of $n$ non-overlapping identical spheres in $\R^d$. Thus the configurations $\bgrande$ that minimise the energy are maximising their contact number $c(\bgrande)$, as claimed.
\end{proof}
The identification of the set of configurations whose contact number equals $\mathsf{c}(n,d)$ is a difficult topic, even in low dimensions $2$ and $3$, as we present in the next examples. 

Recall that $\mathsf{c}(n,d)$ is also the maximum number of edges that a contact graph of $n$ non-overlapping translations of $B(0,1)$ can have in  $\R^d$, see, e.g., the monograph~\cite{B} and~\cite[Chapter 3]{BAK}. Using this approach, we now review in more detail the situation in the Euclidean spaces $\R^2$ and $\R^3$.

\begin{example}[The planar case, $d=2$]

An explicit general formula for $\mathsf{c}(n,2)$ was introduced by Erd\H{o}s in 1946, but rigorously established only in 1974 by Harborth,~\cite{Ha74}: 
\begin{equation*}
    \mathsf{c}(n,2)=\lfloor 3\, n-\sqrt{12n-3}\rfloor, \quad n\geq 2.
\end{equation*}
In particular, one has
\begin{equation*}
    \mathsf{c}(2,2)=1, \quad \mathsf{c}(3,2)=3, \quad \mathsf{c}(4,2)=5,\quad \mathsf{c}(5,2)=7,\quad \mathsf{c}(6,2)=9, \quad \mathsf{c}(7,2)=12.
\end{equation*}
Moreover, the hexagonal packing arrangement is a cluster configuration whose contact number achieves $\mathsf{c}(n,2)$ for all $n$, see Figure~\ref{fig:c(n,2)}. But it also corresponds to clusters which realise the densest packing.
Nevertheless, the problem of recognising (all) contact graphs of unit disc packing is NP-hard, see~\cite{BK95}.

\begin{figure}
\begin{tikzpicture} [scale=0.5]
   \foreach \th in {120} 
      {\draw (0,0) circle(1); \draw (0,0) +(\th:2) circle(1); 
       \draw (2,0) circle(1); \draw (2,0) +(\th:2) circle(1); 
       \path[draw=red] (0,0) -- (2,0) ; \path[draw=red] (0,0) -- +(\th:2) ; 
       \path[draw=red] (2,0) +(\th:2) -- (2,0) ; 
       \path[draw=red] (2,0) ++(\th:2) -- +(-2,0)  ; 
        };
   \path[draw=red] (0,0) -- +(60:2) ;  
\end{tikzpicture}
\hspace{1cm}
\begin{tikzpicture} [scale=0.5]
\draw (0,0) circle(1); 
\foreach \i in {0,-60,...,-180}{ \draw (\i:2) circle(1); \path[draw=red] (\i:2) -- (0,0); }; 
\path[draw=red] (0:2) { \foreach \i in {0,-60,...,-180}{ -- (\i:2)} };
\end{tikzpicture}
\hspace{1cm}
\begin{tikzpicture} [scale=0.5]
\draw (0,0) circle(1); 
\foreach \i in {0,-60,...,-240}{ \draw (\i:2) circle(1); \path[draw=red] (\i:2) -- (0,0); }; 
\path[draw=red] (0:2) { \foreach \i in {0,-60,...,-240}{ -- (\i:2)} };
\end{tikzpicture}
\hspace{1cm}
\includegraphics[scale=0.09]{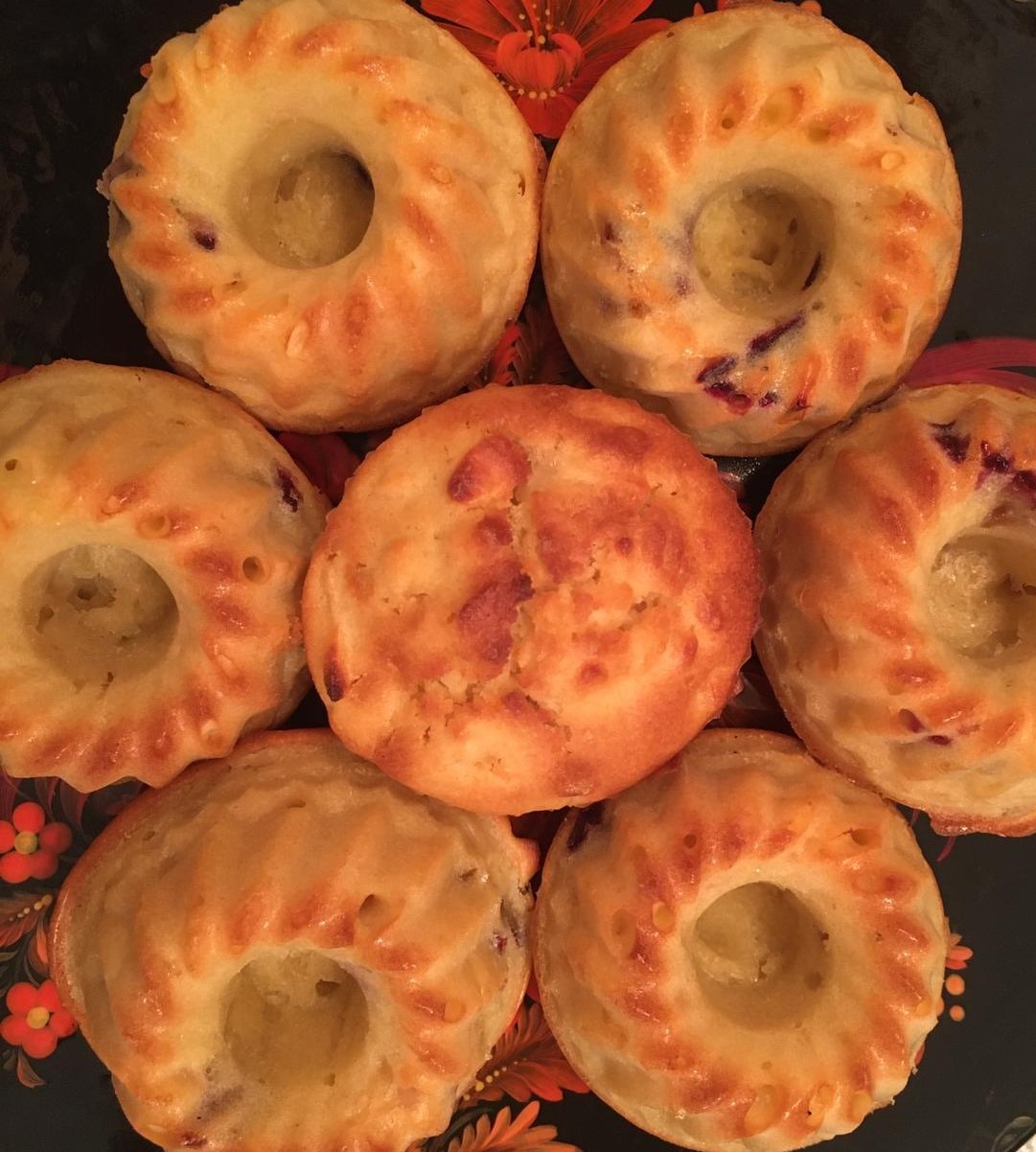}
\vspace{10pt}
\caption{Planar configurations with minimal energy $\nrg_n^*$ for $n=4,5,6$ and their induced contact graphs with respectively $5,\,7,\,9$ pair contacts. Optimal Kugelhopf configuration for $n=7$.}
\label{fig:c(n,2)}
\end{figure}
We now study the asymptotic behaviour of the minimal energy $\nrg_*$ for small $\rho$. 
It means that we expand the maximal overlap area obtained in \eqref{eq:maxVovlapd2}
$$ \Vovlap^*
   = 2\rdep\rspace ^2 \Big( \arccos\big(\frac{1}{1+\rho}\big) - \frac{1}{(1+\rho)^2}\sqrt{\rho (2 + \rho)}\, \Big) $$
   as a function of $\rho$.
Using the expansion for $\rho \ll 1$ of the function $\arccos$ around the point $1$, 
$\arccos(1-x) = \sqrt{2x} \big(1 + \frac{1}{12}x + O(x^2)\big)$, we obtain 
\begin{eqnarray*}
\Vovlap^*
  &=& 2\rgrande\rspace ^2 \Big( (1 + \rho)^2 \sqrt{2\frac{\rho}{1+\rho}}\big(1 + \frac{1}{12}\frac{\rho}{1+\rho} + O(\rho^2) \big) - \sqrt{2\rho} \sqrt{1 + \rho/2}\, \Big)\\
  &=& 2\rgrande\rspace ^2 \sqrt{2\rho} \Big( (1 + 2\rho)(1-\frac{1}{2}\rho)(1 + \frac{1}{12}\rho) - (1 + \frac{1}{4}\rho)\, \Big) + O(\rho^{5/2})\\
  &=& \frac{8\sqrt{2}}{3} \rgrande\rspace ^2 \, \rho^{3/2}+ O(\rho^{5/2}).
\end{eqnarray*}
The minimal energy $\nrg_n^*$ is then given by
\begin{equation*}
    \nrg_n^* = n \,\pi \,\rgrande\rspace ^2 \, \Big( 1 + 2 \rho  -  \frac{8\sqrt{2}}{3\pi}\frac{\mathsf{c}(n,2)}{n}  \, \rho^{3/2}+ O(\rho^2) \Big).
\end{equation*}
\end{example}

\begin{example}[The $3$-dimensional case]
The theoretical situation in $\R^3$ is mainly unsolved, with no hope of obtaining an explicit general formula for $\mathsf{c}(n,3)$ in terms of $n$. The only known exact values are the trivial ones, that is:
\begin{equation*}
    \mathsf{c}(2,3)=1,\quad \mathsf{c}(3,3)=3,\quad \mathsf{c}(4,3)=6, \quad \mathsf{c}(5,3)=9.
\end{equation*}
The largest known number of contacts for $n=6,7,8,9$ are $12,15,18,21$, respectively, see~\cite{BAK}. 
They correspond to clusters which also satisfy the property of being densest packing: the centres of the spheres are the lattice points of a face-centred cubic lattice, see~\cite[Figure 1]{BAK}.

In~\cite{HHGOH}, the authors study \emph{numerically} the number and the structure of optimal finite-sphere packings -- called isoenergetic states -- via exact enumeration. They underline the implications for colloidal crystal nucleation.
\medbreak

To conclude this section, we recall the expression of the  maximal overlap volume $\Vovlap^*$ computed in  Example \ref{ex:depinterdR3}:
\begin{equation*}
    \Vovlap^*= 
2\pi \, \rgrande\rspace ^3 \, \rho^2 \big(1 + \frac{2}{3} \rho\big).
\end{equation*}
Therefore, by \eqref{minnrg}, 
the minimal energy $\nrg_n^*$ is given by
\begin{equation*}
    \nrg_n^* = n \, \frac{4}{3} \,\pi \, \rgrande\rspace ^3 (1+\rho)^3  - \mathsf{c}(n,3) \, 2\pi \, \rgrande\rspace ^3 \, \rho^2 \big(1 + \frac{2}{3} \rho\big),    
\end{equation*}
which leads to the following  asymptotic expansion in $\rho$:
\begin{equation*}
\nrg_n^* = n \, \frac{4}{3} \,\pi \, \rgrande\rspace ^3  \Big( 1 \, + 3 \rho + 3 \Big(1 - \frac{\mathsf{c}(n,3)}{2n}\Big) \rho^2 + O(\rho^3)\Big) , \ \rho \ll 1.
\end{equation*}
\end{example}

\section{Numerical simulations}\label{AppendixSimulations}
We created a companion GitLab page (\href{https://lab.wias-berlin.de/zass/dynamics-of-spheres}{{\small \tt https://lab.wias-berlin.de/zass/dynamics-of-spheres}}), where the reader can find the code we wrote to generate illustrative movies that simulate various planar random dynamics of the type studied in this paper. In particular, we propose illustrations of Section \ref{sect_dyninfinie}, Section \ref{sec:dep_dynamics}, 
and Section \ref{sec:packing}, that is:
\begin{itemize}
    \item The two-type dynamics \eqref{eq:Sninfty} of 
    Brownian hard spheres in a Brownian particle bath and their projections on the hard sphere system, for different values of $\rho$ and $\zpiccola$. The short range of the depletion attraction is clearly visible.
    \item The one-type dynamics \eqref{eq:Sndep} of 
    Brownian hard spheres submitted to a gradient field derived from the depletion interaction.
    \item The dynamics \eqref{eq:Sninfty} of 
    Brownian hard spheres in a Brownian particle bath with high density $\zpiccola$, which induces a very high stability of the initial hard sphere cluster configuration, close to a maximiser of its contact number.
\end{itemize}

\section*{Acknowldegments}
MF acknowledges support from the Labex CEMPI (ANR-11-LABX-0007-01). 
JK was partially funded by Deutsche Forschungsgemeinschaft (DFG) through grant CRC 1114 Scaling Cascades in Complex Systems, Project no. 235221301, Project C02 \emph{Interface dynamics: Bridging stochastic and hydrodynamic descriptions}.
SR thanks her Alsatian grandmother for the jelly doughnut recipe. 
AZ was funded by DFG through DFG Project no. 422743078 \emph{The statistical mechanics of the interlacement point process} in the context of SPP 2265 Random Geometric Systems.

\bigbreak


\end{document}